\documentclass[reqno, 12pt]{amsart}
\usepackage{mathrsfs}
\usepackage{mathrsfs}
\usepackage{amsfonts}
\usepackage{bbm}
\usepackage{bbm}
\usepackage{mathrsfs}
\usepackage{amssymb,latexsym}
\newtheorem{theorem}{Theorem}

\newtheorem{remark}[theorem]{Remark}
\newtheorem{proposition}[theorem]{Proposition}
\newtheorem{lemma}[theorem]{Lemma}

\newtheorem{corollary}[theorem]{Corollary}

\begin{document}
\setlength{\baselineskip}{1.2\baselineskip}

\title [Equivariant Hopf bifurcation with arbitrary pressure laws]
{Equivariant Hopf bifurcation with general pressure laws}

\author{Tong Li$^\star$}
\address{$^\star$Department of Mathematics\\
        The University of Iowa, Iowa City\\
         IA, 52242, USA}
\email{tong-li@uiowa.edu}

\author{Jinghua Yao$^\dag$}
\address{$^\dag$Corresponding author. Department of Mathematics\\
        The University of Iowa, Iowa City\\
         IA, 52242, USA}
\email{jinghua-yao@uiowa.edu}

\maketitle

\begin{abstract}
The equivariant Hopf bifurcation dynamics of a class of system of partial differential equations is carefully studied.
The connections between the current dynamics and fundamental concepts in hyperbolic conservation laws are explained. The unique approximation
property of center manifold reduction function is used in the current work to determine certain parameter in the normal form. The current work generalizes the study of the second
author ([J. Yao, $O(2)$-Hopf bifurcation for a model of cellular shock instability, Physica D, 269 (2014), 63-75.]) and supplies a class of examples of $O(2)$ Hopf bifurcation with two parameters arising
from systems of partial differential equations.



\end{abstract}

\maketitle

\section{Introduction}\label{introduction}

\setcounter{equation}{0}
\setcounter{theorem}{0}

In this paper, we generalize the former study of the second author in Yao \cite{Yao} for the equivariant Hopf bifurcation driven by partial differential equations in two aspects: on one hand, we generalize the study in Yao \cite{Yao} to arbitrary nonzero wave numbers and general fluxes, on the other hand, we treat the variations of two parameters in the system after scalings and renaming in one go. The former generalization enables us to find connections of our study with fundamental concepts in the theory of hyperbolic conservation laws while the latter generalization enables us to treat the different effects together. Therefore, the current work makes the effects of different terms for the current dynamics more transparent and supplies a class of systems of partial differential equations which undergo $O(2)$ Hopf bifurcation with two bifurcation parameters and with connection with hyperbolic conservation laws. Later, it will be clear that both aspects are very subtle and nontrivial. The symmetry due to the structure of the system under consideration makes the generalization in the former aspect possible while the idea of using unique approximation property of the center manifold reduction function to determine certain coefficient(s) in normal form makes it possible for us to deal with multidimensional bifurcation parameters in the current work. During the process of carrying out the current generalization, other important points in utilizing the center manifold theory and normal form theory will be emphasized. Another point of the current work is to clarify some possible ambiguity and certain misprints/typos in the work Yao \cite{Yao}. 

To proceed, let us first recall the general form of systems studied in Yao \cite{Yao} as follows
\begin{equation}\label{system}
\begin{cases}
\partial_t \tau-\partial_x u=-a\partial^4_x\tau,\\
\partial_t u-\partial_x\sigma(\tau)=-\delta\partial^2_x
u-\varepsilon\partial^4_x u
\end{cases}
\end{equation}
on the spatial periodic domian $\mathbb{R}^1/[-M, M]$ where $M$ is an arbitrary positive constant.
Here $\tau=\tau(x,t)$ and $u=u(x,t)$ are the scalar unknown functions of space variable $x$ and time variable $t$. The scalar function $\sigma(\tau)$ is usually called
flux function in mathematics and pressure law in physics. The three parameters $a, \delta,\varepsilon$ are diffusion coefficients. 

The spatial domain is $\mathbb{R}^1/[-M, M]$, which means that we have periodic boundary conditions. It is due to this boundary condition that we can use Fourier analysis to study the spectra of the linear operators involved and compute ingredients in the joint use of center manifold reduction and normal form transformation. Without loss of generality and for the ease of exposition, we can always consider the case $M=\pi$ or else we can do the  following scaling and renaming 
$$t\mapsto \bar{t}=\frac{\pi}{M}t,\,x\mapsto \bar{x}=\frac{\pi}{M}x,\, a\mapsto\bar{a}= \frac{\pi^3}{M^3}a,\, \delta\mapsto\bar{\delta}=\frac{\pi}{M}\delta,\, \varepsilon\mapsto\bar{\varepsilon}= \frac{\pi^3}{M^3}\varepsilon.$$
From now on, we make the convention $M=\pi$. After another scaling and renaming 
 $$t\mapsto\tilde t=\varepsilon t, x\mapsto\tilde x=x, u\mapsto \tilde u=u, \tau\mapsto  \tilde \tau=\varepsilon \tau,$$ 
 $$a\mapsto\tilde a= \varepsilon^{-1}a, \sigma(\tau)\mapsto\tilde \sigma(\tilde\tau)=
 \varepsilon^{-1}\sigma(\varepsilon^{-1}\tilde\tau)=\varepsilon^{-1}\sigma(\tau), 
 \delta\mapsto\tilde\delta=\varepsilon^{-1}\delta,$$ we can always assume $\varepsilon=1$. 
We will validate this convention $\varepsilon=1$ now. 
 
Now we observe that any constant state $(\tau_0, u_0)$ is a solution to system (\ref{system}) mathematically. We can always regard the constant state $(0, 0)$ as a physical meaningful state without loss of generality. This is obvious by first choosing 
$(\tau_0, u_0)$ in the physical range and then using the invariance of system (\ref{system}) in the 
translation group actions $u\mapsto u+h$ for $h\in \mathbb{R}^1$ and redefining 
$\sigma(\tau)$ by $\sigma(\tau_0+\tau)$.   Therefore, we are interested in investigating the interplay between the nonlinearities from the flux functions 
and the competing diffusions in the dynamics of system \eqref{system} and giving a complete characterization of the equivariant Hopf bifurcation dynamics based on arbitrary 
nonzero wave numbers $k_0$ and for general flux functions $\sigma(\tau)$ near the solution $(0,0)$. 

Let us say briefly several words about the physical aspect of our study. Of course, systems of form \eqref{system} \textit{without} fourth order diffusions are generic in classical mechanics (\cite{Ba, Da, FW} 
and the references therein), in gas dynamics, for example the $p$-system (see in particular Chapter 2 of 
Dafermos \cite{Da}). However, systems of the same form \textit{with} fourth order diffusion terms are also of interest and appear frequently.  When $M=\infty$, systems of form \eqref{system} are
also connected with the Kuramoto-Sivashinsky and related systems when we seek traveling waves 
solutions (see \cite{EW, Go, KT, Si, SM, Yao}). Another situation where fourth order terms appear is in the use of the vanishing viscosity method, which is classical in the study of hyperbolic conservation laws.  For us, systems of form \eqref{system} also appear as a class of simplified models related to our study on Hopf bifurcations of shock waves (see \cite{PYZ}).
 
To state our results, we first introduce some notations. The first one relates to symmetry. We say an equation $\partial_t U=F(U)$ 
on a Banach space $\mathcal X$ exhibits \textit{$G$-symmetry} or is \textit{$G$-equivariant} with respect to some isometry 
group $G$ if the vector field $F:\mathcal X\mapsto \mathcal X$ is {$G$-equivariant}, i.e., $F(gU)=gF(U)$ on $\mathcal X$ for any action given by $g\in G$. Another notation is related to the admissibility of parameters $a,\delta$ in our system. To avoid ambiguity, we will use $a_c, \delta_c$ where the subindex ``$c$" indicates ``critical and admissible". For any fixed integer 
$k_0\neq0$, we call the set $\mathcal A(k_0)$ defined by 
\begin{equation*}\Big\{(a_c,\delta_c)\in 
\mathbb{R}^2; 0\neq\delta_c=(a_c+1)k_0^2,\, 
\sigma'(0)>a_c^2k_0^6, a_ck^4(k^2-\delta_c)\neq-\sigma'(0)\,\mbox{for}\, k\neq\pm k_0\Big\}
\end{equation*} the \textit{admissible critical configuration set} associated with the wave number $k_0$. A point $(a_c,\delta_c)\in\mathcal A(k_0)$ 
is called an \textit{admissible critical configuration point}. 

There are some hidden requirements on the one variable function $\sigma(\cdot)$ in the definition of $\mathcal A(k_0)$ such as $\sigma'(0)>0$. As we will also do linearization in our later computations, we need that the function $\sigma(\cdot)$ has certain smoothness. Such a necessary but purely technical assumption will be imposed for $\sigma$ around $0$ for the obvious reason that we do the bifurcation analysis around $(\tau, u)=(0,0)$. For our purpose here, $\sigma(\tau)\in C^3$ around $\tau=0$ will be sufficient. But we will assume that
$\sigma(\tau)$ is $C^{\infty}$ in a neighborhood of $\tau=0$ in $\mathbbm R^1$ to avoid this technical issue. We will validate this assumption from now on. Now we are in a position to state our main results in which $(a_c,\delta_c)$ is understood as any admissible critical configuration point and $\mu_1=a-a_c$, $\mu_2=\delta-\delta_c$.

\begin{theorem}\label{thm}
If $\sigma''(0)\neq0$, then system \eqref{system} or equivalently $\eqref{perturb}$ undergoes an $O(2)$ Hopf bifurcation around $(a,\delta, \tau, u)=(a_c,\delta_c, 0, 0)$ in the Hilbert space consisting of functions in $L^2_{per}(-\pi, \pi)$ with zero mean. There are two families of bifurcated rotating waves and a torus of bifurcated standing waves.
\end{theorem}

The following theorems explain the dynamics of the system \eqref{system} given by Theorem \ref{thm} more precisely.

\begin{theorem}
System \eqref{system} admits a center manifold reduction with $O(2)$ symmetry near $(\mu_1,\mu_2)=(0,0)$.  If the center space is parametrized by $z_1\xi_0+z_1^*\xi_0^*+z_2\xi_1+z_2\xi_1^*$, the dynamics on the center manifolds has the following form
\begin{equation}\label{normal}
\begin{cases}
\frac{d}{dt}z_1=i\omega_c z_1+ z_1\Big(\mathbbm{a}(\mu_1,\mu_2)+\mathbbm{b}_0|z_1|^2+\mathbbm{c}_0|z_2|^2+O(|\mu_1|,|\mu_2|, |z_1|, |z_2|)\Big)\\
\frac{d}{dt}z_2=i\omega_c z_2+z_2\Big(\mathbbm{a}(\mu_1,\mu_2)+\mathbbm{b}_0|z_2|^2+\mathbbm{c}_0|z_1|^2)+O(|\mu_1|,|\mu_2|, |z_1|, |z_2|)\Big)
\end{cases}
\end{equation}
where the real parts of $\mathbbm a(\mu_1,\mu_2)$, $\mathbbm b_0$ and $\mathbbm c_0$ are given by
\begin{equation}
\mathbbm a_r=\frac{k_0^2}{2}(-\mu_1 k_0^2+\mu_2), \mathbbm b_{0r}=-6k_0^6\sigma''(0)^2\delta_c\alpha^{-1}, \mathbbm c_{0r}=0
\end{equation}
respectively. Here $\omega_c=(\sigma'(0)k_0^2-a_c^2k_0^8)^{1/2}>0$ and $\alpha$ is a positive constant depending on $k_0^2$ the specific value of which is not important for us and the term $O(|\mu_1|,|\mu_2|, |z_1|, z_2)$ is a sum of terms $O((|\mu_1|+|\mu_2|)(|z_1^2|+|z_2^2|))$, $O((|\mu_1^2|+|\mu_2^2|)(|z_1^2|+|z_2^2|))$ and $O(|z_1|^4+|z_2|^4)$. 
\end{theorem}

\begin{theorem}
Let $\vartheta$ be defined as $\vartheta=-\mu_1 k_0^2+\mu_2$.
Parameterizing the solution by $\vartheta$, the following properties hold for system \eqref{system} in a small neighborhood of
$(\mu_1,\mu_2)=(0,0)\in\mathbbm R^2$:
\begin{enumerate}
 \item if $\mathbbm b_{0r}<0$, then (i) it has precisely one equilibrium $U_{\vartheta}$ for
$\vartheta<0$ with $U_0=0$ and this equilibrium is stable; (ii) it possesses for $\vartheta>0$
equilibria $U_{\vartheta}$, bifurcated rotating waves and bifurcated standing waves.
Both the rotating waves and the standing waves have amplitudes $\mathcal{O}(|\vartheta|^{1/2})$ or equivalently $\mathcal{O}((|\mu_1|+|\mu_2|)^{1/2})$.
The equilibria and rotating waves are unstable and the standing waves  are stable.
\item if $\mathbbm b_{0r}>0$, then (i) it has precisely one equilibrium $U_{\vartheta}$ for
$\vartheta>0$ with $U_0=0$ and this equilibrium is unstable; (ii) it possesses for $\vartheta<0$
equilibria $U_{\vartheta}$, bifurcated rotating waves and bifurcated standing waves.
Both the rotating waves and the standing waves have amplitudes $\mathcal{O}(|\vartheta|^{1/2})$.
The standing waves, the rotating waves are unstable and the equilibria  are stable.
\end{enumerate}
\end{theorem}

The current work in particular enables us to draw the following conclusions given below.

The condition $\sigma''(0)\neq0$ means that the genuine nonlinearity of the characteristic fields of the corresponding
first order system in system \eqref{system} is essential in order to have non-degenerate $O(2)$ Hopf bifurcation dynamics. In view of the fact $\delta_c\neq0$ from the definition of $\mathcal A(k_0)$,we know that the absence of second order diffusion term or the loss of
genuine nonlinearity leads to the degeneracy of the $O(2)$ Hopf bifurcation dynamics. More explanations are given in Section \ref{spectra}.

By examining the expressions of $\mathbbm a$, $\mathbbm b_0$, $\mathbbm c_0$ (see \eqref{mathbbma}, \eqref{mathbbmbr}, \eqref{mathbbmbi} and \eqref{mathbbmc}) and $\omega_c$, we see that they are all functions of $k_0^2$. This fact embodies the symmetry of the system \eqref{system}. The two bifurcation parameters $\mu_1$ and $\mu_2$ determine the dynamics through a combination $-\mu_1 k_0^2+\mu_2$. This is due to that fact that $\mu_1$ enters the bifurcation equation \eqref{bs} through the operator $-\mu_1\partial_x^4$ which is two order higher than $-\mu_2\partial_x^2$ in differentiation.

Another interesting and surprising phenomenon is that $\sigma''(0)$ enters the determining parameters $\mathbbm b_0$ and $\mathbbm c_0$ through $\sigma''(0)^2$ while $\sigma'''(0)$ only appears in $\mathbbm c_0$. The direct consequences are that the stability of the bifurcated waves does not depend on the sign of $\sigma''(0)$ as long as it does not vanish and that $\sigma'''(0)$ only enters the angular equations in the reduced dynamics. Of course, $\sigma'(0)$, $\sigma''(0)$ and $\sigma'''(0)$ come from the expansion of the flux function $\sigma(\tau)$ around zero and correspond to linear term, second and third order nonlinear terms in the bifurcation equation \eqref{bs}. It is well-known that $\sigma'(0)>0$ and $\sigma''(0)\neq0$ correspond to hypebolicity and genuine nonlinearity for the first order system around $\tau=0$. However, we do not know at this writing what is the counterpart in hyperbolic conservation laws for $\sigma'''(\tau)$.

In the current work, we supply a class of systems from a system of partial differential equations which undergo $O(2)$ Hopf bifurcation with two bifurcation parameters involved. By our choice of working spaces and definition of admissible critical configuration points, we excluded the possibility that $0\in\sigma(\mathcal L_c)$ (see Section \ref{spectra} for notations). As a consequence, we do not consider here the so-called zero-Hopf bifurcation scenario in which the spectral set of the linear operator involved should contain the number zero. However, we remark that the zero-Hopf bifurcation itself is very interesting.

The former study of the second author in \cite{Yao} corresponds to a very special case in which $\sigma(\tau)=1+c^2\tau+\tau^2$, $k_0=1$, $\mu_1=0$, $\mu_2=\mu\in\mathbbm R^1$ and $(a,\delta)$ is understood as admissible critical configuration points. If we make these identifications, the coefficients $\mathbbm a$ and $\mathbbm c_0$ here reduce to  $\mathbbm a$ and $\mathbbm c$ in \cite{Yao} respectively. However, there is an error in the expression of $\mathbbm b$ in \cite{Yao} which we would like to correct here though it does not affect the conclusions there at all. We missed the complex unit $i$ in the computation of $\mathbbm b$ in page 72 of \cite{Yao} and we made a mistake when we tried to get the real and imaginary parts of
$\mathbbm b$ in page 73 of \cite{Yao} though we computed $\mathbbm b$ correctly. The correct $\mathbbm b$ in \cite{Yao} should be
$\mathbbm b=\frac{-6(a+1)-\omega^{-1}(48a-15a^2)i}{(48a-15a^2)^2+36(a+1)^2\omega^2}$, which is consistent with the $\mathbbm b_0$ here.
Meanwhile, the space $Y$ in \cite{Yao} consist of functions in $H^2_{per}(-\pi,\pi)$ with mean zero, which gave better result as the bifurcation dynamics occurs in $Y$. The reason is that
the nonlinearity $R(U,\mu)$ in \cite{Yao} verifies the assumptions for the nonlinearities (i.e., (1)) in the center manifold theorem Theorem \ref{center manifold theorem} with such a choice. If $\mu_1\neq 0$, we can only choose $Y$ as in the current work as fourth order derivative involves in $R(U,\mu_1,\mu_2)$ here. The extension from one bifurcation parameter to two bifurcation parameters also bring us great computational complexity and differences. Our strategies to get the final preferred normal form with determined coefficients could serve as an example to deal with bifurcations with more than one parameters.

For results using  center manifold theory, see Bressan \cite{Br1}, Carr \cite{Ca}, Henry \cite{He}, Haragus and Iooss \cite{HI}, Chicone \cite{Ch}, Golubitsky-Stewart-Schaeffer \cite{GSS}, Iooss and Adelmeyer \cite{IA}, Wittenberg and Holmes \cite{WH} and the references therein. In particular, see the works Bianchini and Bressan \cite{BB}, Texier and Zumbrun \cite{TZ1, TZ2} in the study of conservation laws and viscous traveling waves, see Nakanishi and Schlag \cite{NS} for applications in the dispersive equations.

We organize the remaining part of the current paper as follows: we do spectral analysis in Section \ref{spectra} and study symmetry in Section \ref{symmetry} while computations and analysis are done in
Section \ref{analysis}.

\textbf{Convention.} We set $\varepsilon=1$ and $M=\pi$ and $\mathcal L_c=\mathcal L(a_c,\delta_c)$ for convenience. We will use ``$*$" to denote
``complex conjugate", i.e., for $z\in \mathbb{C}$, $z^*$ means the
complex conjugate of $z$; for complex numbers, we use subindices $r$ and $i$ to indicate their real and imaginary parts respectively; for two
nonnegative quantities, ``$A\lesssim B$" means ``$A\leq CB$" for
some constant $C>0$. For $U=\begin{pmatrix}u_1\\u_2\end{pmatrix},
V=\begin{pmatrix}v_1\\v_2 \end{pmatrix}\in \mathbb{C}^2$, ``$\langle
U, V \rangle" \,\mbox{means}\, ``\int_{-\pi}^{\pi} u_1 v_1^* + u_2 v_2^*\,dx$". We do not distinguish $(\tau, u)$ and $\begin{pmatrix}\tau\\u\end{pmatrix}$.
We also adopt the standard big $O$ and small $o$ notations for limiting processes. This should not be confused with the group $O(2)$. For a vector $V$, we use $v_j$ or $v^{(j)}$ to represent its components. We use ``$[\cdot\,\, , \cdot ]$" to denote commutator: $[F, G]=FG-GF$ for F, G being functions, symbols or operators.
For a linear operator $\mathcal{L}: \mathcal X\mapsto \mathcal X$ on some Banach space $\mathcal X$, we use $\sigma(\mathcal{L})$, $\rho(\mathcal{L})$ to denote
its spectral set and resolvent set. Further, $\sigma(\mathcal{L})=\sigma_c(\mathcal{L})\cup\sigma_s(\mathcal{L})\cup\sigma_u(\mathcal{L})$, i.e.,
$\sigma(\mathcal{L})$ is the union of the center spectral set $\sigma_c(\mathcal{L})$, the stable spectral set $\sigma_s(\mathcal{L})$ and the unstable
spectral set $\sigma_u(\mathcal{L})$. The hyperbolic space is given by $\mathcal X_h:=\mathcal X_s \cup \mathcal X_u$. 

\section{Spectra}\label{spectra}

\setcounter{equation}{0}
\setcounter{theorem}{0}

Now we regard $(\tau, u)$ as the perturbation variables around the state $(0, 0)$ and get the following nonlinear perturbation system by Taylor expansion
\begin{equation}\label{perturb}
\begin{cases}
\partial_t \tau-\partial_x u=-a\partial^4_x\tau\\
\partial_t u-\sigma'(0)\partial_x\tau=-\delta\partial^2_x u-\partial^4_x u
+\partial_x\Big( \frac{\sigma''(0)}{2}\tau^2 +  \frac{\sigma'''(0)}{6}\tau^3 +\Gamma(\tau)\Big)
\end{cases}
\end{equation}
where $\Gamma(\tau):=\sigma(\tau)-\sigma(0)- \frac{\sigma''(0)}{2}\tau^2 - \frac{\sigma'''(0)}{6}\tau^3$ and $\Gamma(\tau)=O(|\tau|^4)$ when $|\tau|$ is small. 
We use the identifications $U=\begin{pmatrix} \tau \\ u
\end{pmatrix}$, $\mathcal{L}(a,\delta)=\begin{pmatrix}-a\partial_x^4 & \partial_x \\
\sigma'(0)\partial_x & -\delta\partial_x^2-\partial^4_x
\end{pmatrix}$ and $\mathcal{N}\begin{pmatrix} \tau \\ u
\end{pmatrix}=\begin{pmatrix} 0\\ \partial_x\Big( \frac{\sigma''(0)}{2}\tau^2 +  \frac{\sigma'''(0)}{6}\tau^3 +\Gamma(\tau)\Big)
\end{pmatrix}$ to write the nonlinear perturbation system \eqref{perturb} into an operator equation $\partial_t U=\mathcal{L}(a,\delta)U+ \mathcal{N}(U)$.

Now we proceed to fix the working spaces. Here $\mathcal L$ is a fourth order linear differential operator on the spatial periodic domain $\mathbb{R}^1/[-\pi, \pi]$ and we need to work on a space triplet involving $L^2_{per}(-\pi, \pi)$. For our specific purpose in the current work, we will work with the space triplet $Z\subset Y\subset X$ given below:
$$Y=X=\{U\in L^2_{per}(-\pi, \pi); \frac{1}{2\pi}\int_{-\pi}^{\pi}U\,dx=0\},$$
$$Z=\{U\in H^4_{per}(-\pi, \pi); \frac{1}{2\pi}\int_{-\pi}^{\pi}U\,dx=0\}.$$
The above choice is not at random and has important implications. In certain cases, the above choice of space triplet does not give best results. 
\begin{remark}
The mean zero restriction comes naturally from the conservative form of system \eqref{system}.  In fact, if we seek solutions on the periodic 
Sobolev spaces, any solution $(\tau, u)$ is conserved due to the conserved form of the original system \eqref{system} as follows
\begin{equation*}
\partial_t\begin{pmatrix} \tau \\ u
\end{pmatrix}=\partial_x\begin{pmatrix}u-a\partial_x^3  \\
\sigma(\tau)-\delta \partial_x u-\partial_x^3 u
\end{pmatrix}.
\end{equation*}
\end{remark}

Now we explain the definition of $\mathcal A(k_0)$ and assumptions on $\sigma(\cdot)$. For this purpose, we first analyze the spectrum of the linear differential operator $\mathcal L(a,\delta)$ considered on $X$ with domain $Z$. To this end, we can proceed by Fourier analysis as we are working on periodic domains.
After Fourier transformation, the differential operator is represented by
$$M_k=\begin{pmatrix}-ak^4 & ik \\
\sigma'(0)k i & \delta k^2-k^4
\end{pmatrix},\,\, k\in\mathbb{Z}, \, k\not=0.$$
Therefore, we have $\sigma(\mathcal{L})=\cup_{k\in\mathbb{Z}, \, k\not=0}\sigma(M_k)$.
The mode $k=0$ is not included in the above union due to the mean zero restriction in the definition of $X$.
The eigenvalues of $M_k$ for $k\not =0$ are given by 
\begin{equation}\label{character}
\det (\lambda-M_k)=\lambda^2+\Big( (a+1)k^4-\delta k^2 \Big)\lambda+ a k^4(k^4-\delta k^2)+\sigma'(0)k^2=0.
\end{equation}

To expect Hopf bifurcation, we need for $k_0\not=0$, $a_c$ and $\delta_c$ that
$$(a_c+1)k_0^4-\delta_0 k_0^2=0,\,\, \mbox{and}\,\, a_c k_0^4(k_0^4-\delta_c k_0^2)+\sigma'(0)k_0^2>0,$$
which is equivalent to
\begin{equation}\label{condition}
(a_c+1)k_0^2-\delta_c=0,\,\, \mbox{and}\,\, \sigma'(0)-a_c^2k_0^6>0.
\end{equation}

The condition \eqref{condition} means that $M_{k_0}$ and $M_{-k_0}$  contributes a pair of complex conjugate eigenvalues to $\sigma_c(\mathcal L)$. Meanwhile, we notice that the characteristic equations of $M_k$ and $M_{-k}$ coincide as $k$ enters these equations through $k^2$. Later, we will see that this is due to the $O(2)$-symmetry of our system. Also, for any $a_c$ and $\delta_c$ satisfying \eqref{condition}, the coefficient $(a_c+1)k^4-\delta_c k^2$ of the first order term in $\lambda$ in \eqref{character} never vanishes for any nonzero $k$ such that $|k|\not=|k_0|$. Therefore, there are no other pairs of complex conjugate eigenvalues  of $\mathcal L(a_c, \delta_c)$. The remaining condition that $a_ck^4(k^2-\delta_c)\not=-\sigma'(0)$ for any nonzero integer  $k\neq \pm k_0$ in the definition of $\mathcal A(k_0)$ enables us to maintain an $O(2)$ Hopf bifurcation scenario by excluding the possibility that $0\in \sigma(\mathcal L(a,\delta))$.  By the above analysis, we get the following proposition:
\begin{proposition}\label{cp}
$\sigma_c(\mathcal L_c)=\{i\omega_c, -i\omega_c\}$ where $\omega_c=\omega(a_c, \delta_c, k_0):=(\sigma'(0)k_0^2-a_c^2k_0^8)^{1/2}>0$.
\end{proposition}

\begin{remark}
The mean zero requirement in the current space triplet is crucial not only in excluding the possibility $0\in \sigma_0(\mathcal L_c)$ to guarantee the equivariant Hopf bifurcation spectrum scenario but also in validating the resolvent estimate in the center manifold theorem to guarantee the existence of center manifold. See Remark \ref{gcm}.
\end{remark}

Further, we have the following simple spectral gap lemma concerning the spectrum of $\sigma(\mathcal L_c)$ which is needed to verify the assumptions in center manifold theory.

\begin{lemma}\label{gap}
There exists a positive constant $\gamma>0$, such that $\sup\{Re\, \lambda; \lambda\in \sigma_s(\mathcal L_c)\}<-\gamma$ and  $\inf\{Re\, \lambda; \lambda\in \sigma_u(\mathcal L_c)\}>\gamma$.
\end{lemma}
\begin{proof}
We need to consider the distribution of the roots of the equations
$$\lambda^2+\Big( (a_c+1)k^4-\delta_c k^2 \Big)\lambda+ a_c k^4(k^4-\delta_c k^2)+\sigma'(0)k^2=0$$
for nonzero $k\neq\pm k_0$. By symmetry, we just need consider the case $|k_0|\neq k\in\mathbb N$. 
By Proposition \ref{cp}, we know that roots of the above equation for $k\neq |k_0|$ do not lie in the imaginary axis. Hence we just need to make sure that there is no accumulation of spectra to the imaginary axis in the limit $k\rightarrow +\infty$. This is obvious by writing down the solutions explicitly: the real parts of the roots can only tend to $\pm\infty$. 
\end{proof}

\begin{remark}
Due to the mean zero assumption in the space triplet and the definition of admissible critical configuration point, the so-called \textit{zero-Hopf} bifurcation can not happen for the obvious reason that $0$ is not in  $\sigma(\mathcal L_c)$. It is very interest to make a study under the spectral scenario for zero-Hopf bifurcation.
\end{remark}

Before proceed further and to be rigorous, we claim that $\mathcal A(k_0)\neq\emptyset$ for any $k_0\neq0$ with suitable $\sigma(\tau)$ and $\mathcal \cup_{k_0\neq 0}A(k_0)$ contains all the admissible critical configuration points for the equivariant Hopf bifurcation. The latter part in the claim is trivial. The validity of former part is achieved by simple counting. However, the following analysis is useful for us to have a good feeling about the requirements. For $(a_c,\delta_c)$ to be an admissible critical configuration point, we need that $\delta_c=(a_c+1)k_0^2$, $\sigma'(0)>a_c^2k_0^6$ and there holds $a_c k^4(k^2-(a_c+1)k_0^2)\neq-\sigma'(0)$ for any integer $k\neq k_0$. For any given $\sigma'(0)>0$ and $k_0\neq0$, we have the following observations: 

(1) We see easily that there are infinitely many paris of $(a_c,\delta_c)$: first pick $a_c>0$ so small that $\sigma'(0)>a_c^2k_0^6$ and $a_ck_0^2<1$, then we have $a_c k^4(k^2-(a_c+1)k_0^2)>0>-\sigma'(0)$ for $|k|>|k_0|$, finally adjust $a_c>0$ so small such that the following finite number of relations hold: $a_ck^4(k^2-(a_c+1)k_0^2)>-\sigma'(0)$ for $1\leq k\leq k_0-1$. 

(2) For the case $a_c=0$, we have $\delta_c=k_0^2$ and it is of course true that $a_c k^4(k^2-(a_c+1)k_0^2)=0\neq-\sigma'(0)<0$. 

(3) For the case $a_c<0$, we can first choose $|a_c|$ small so that $\sigma'(0)>a_c^2k_0^6$ and $|a_ck_0^2|<1$, then we have  $a_c k^4(k^2-(a_c+1)k_0^2)\neq-\sigma'(0)$ for $1\geq |k|\leq k_0-1$ and there exists a positive integer $l_0$ such that $a_c k^4(k^2-(a_c+1)k_0^2)<-\sigma'(0)$ for $|k|\geq l_0$. It is possible that $a_c k^4(k^2-(a_c+1)k_0^2)\neq-\sigma'(0)$ for some $k_0+1\leq|k|\leq l_0-1$. All in all, the set $\{a_c k^4(k^2-(a_c+1)k_0^2); k\in \mathbb N, k\neq k_0\}$ is at most countable. However this is not a problem since all we care is if we have an admissible critical configuration point around which we can proceed our bifurcation analysis.

Now, let us explain how $\sigma(\tau)$ comes into play. For this purpose, let us first do some simple computations regarding the first order system
\begin{equation}\label{first}
\begin{cases}
\partial_t \tau-\partial_x u=0,\\
\partial_t u-\partial_x\sigma(\tau)=0.
\end{cases}
\end{equation}
The Jacobian matrix of the flux function $F(U)=\begin{pmatrix}-u\\ -\sigma(\tau)\end{pmatrix}$ with respect to $U=\begin{pmatrix}\tau\\ u\end{pmatrix}$ is
$D_{U}\begin{pmatrix}-u\\-\sigma(\tau)\end{pmatrix}=\begin{pmatrix}0 &- 1 \\
-\sigma'(\tau)& 0
\end{pmatrix}$. If $\sigma'(\tau)>0$ for all $\tau$ under consideration, we get the characteristic pairs of \eqref{first} as follows:
$$\lambda_1(\tau, u)=-\sqrt{\sigma'(\tau)},\, V_1(\tau, u)=\begin{pmatrix}1\\ \sqrt{\sigma'(\tau)}\end{pmatrix},$$
$$\lambda_2(\tau, u)=\sqrt{\sigma'(\tau)},\, V_2(\tau, u)=\begin{pmatrix}1\\ -\sqrt{\sigma'(\tau)}\end{pmatrix}.$$
Easy computations also show that 
$$\nabla_U \lambda_1(\tau,u)\cdot V_1(\tau, u)=-\frac{\sigma''(\tau)}{2\sqrt{\sigma'(\tau)}},\,\,\nabla_U \lambda_2(\tau,u)\cdot V_2(\tau, u)=\frac{\sigma''(\tau)}{2\sqrt{\sigma'(\tau)}}.$$
We wil see that $\sigma'(0)$, $\sigma''(0)$, $\sigma'''(0)$ explicitly come into play in the current work. The condition $\sigma'(0)>a_c^2k_0^6$ in \eqref{condition} forces $\sigma'(0)>0$. Hence $\sigma'(0)>0$ is necessary for equivariant Hopf bifurcation to occur. Interestingly, the positivity of $\sigma'(0)$ corresponds to the \textit{hyperbolicity} of the first order system \eqref{first} 
at $\tau=0$. By continuity, we would have $\sigma'(\tau)>\frac{\sigma'(0)}{2}>0$ if $|\tau|$ is small in amplitude. Hence hyperbolicity remains true under the assumption $\sigma'(0)>0$ if $|\tau|$ remains small.  Further, if hyperbolicity is valid, the condition $\sigma''(\tau)\not=0$ for all $\tau$ under consideration means that the two characteristic fields are \textit{genuinely nonlinear}. Still, by continuity, the genuine nonlinearity retains if $|\tau|$ remains small. We refer the reader to Bressan \cite{Br2} and Lax \cite{Lax} for concepts in hyperbolic conservation laws. After we get the dynamics on the center manifold(s), we will also see the interesting fact that the non-degeneracy of the reduced $O(2)$ Hopf bifurcation dynamics corresponds to the genuine nonlinearity of the system \eqref{first}. It is partially due to the connections of the dynamics with these fundamental concepts in hyperbolic conservation laws that makes our seemingly trivial generalization from the work Yao \cite{Yao} interesting. One reason is that we can not make such a conclusion by studying a specific flux function. In our work here, $\sigma'''(0)$ does also play an important role in our study: it enters the angular (but not the radial) equations of the $O(2)$ Hopf bifurcation dynamics. However, we do not know currently its connection with hyperbolic conservation laws.

Denote $\mu_1=a-a_c$ and $\mu_2=\delta-\delta_c$. To do the bifurcation analysis, we seek to isolate these bifurcation parameters. Hence we write system \eqref{perturb} in the following form
\begin{equation}\label{bs}
\partial_t U=\mathcal L(a_c,\delta_c)U+\mathcal L(a, \delta)U-\mathcal L(a_c,\delta_c)U+\mathcal N (U)\\
\end{equation}
With the identifications $R(U)=R_{11}(U)+R_{20}(U, U)+R_{30}(U, U, U)+\tilde R (U)$ and
$$R_{11}(U)=\begin{pmatrix}-\mu_1\partial_x^4 U^{(1)} \\ -\mu_2\partial_x^2 U^{(2)}\end{pmatrix},\, R_{20}(U, V)=\begin{pmatrix}0 \\ \frac{\sigma''(0)}{2}\partial_x(U^{(1)}V^{(1)})
\end{pmatrix},$$ 
$$R_{30}(U, V, W)=\begin{pmatrix}0 \\ \frac{\sigma'''(0)}{6}\partial_x(U^{(1)}V^{(1)}W^{(1)})
\end{pmatrix},\,\, \tilde{R}(U)=\begin{pmatrix}0 \\ \partial_x \Gamma(U^{(1)})
\end{pmatrix}, $$
we can write the nonlinear perturbation system \eqref{bs} as
\begin{equation}\label{bifurcation equation}
\partial_t U=\mathcal L(a_c,\delta_c)U + R(U, \mu_1, \mu_2).
\end{equation}
Notice that dependence of $R(U)=R(U,\mu_1,\mu_2)$ on $\mu_1$ and $\mu_2$ is  only through the linear term $R_{11}(U)=R_{11}(U, \mu_1$, $\mu_2)$ in the summands of the nonlinear term $R(U)$.
The role of this point lies in reducing a little bit the complications of our computations when we compute the reduced dynamics on the center manifold. However, we still need to be aware of the complications in later computations due to the fact that there are two parameters involved now. Our bifurcation analysis will be done around an arbitrary but fixed admissible critical configuration point $(a_c,\delta_c)$.

\section{Symmetry}\label{symmetry}

\setcounter{equation}{0}
\setcounter{theorem}{0}

We begin this section with the computation of the center space $Z_c$ (or equivalently $X_c$ as they are both finite dimensional) of $\mathcal L_c$. We have the following proposition:

\begin{proposition}\label{parametrization}
The center space $Z_c$ of $\mathcal L_c$ is spanned by $\xi_0=\xi_0(k_0)$, $\xi_1=\xi_1(k_0)$ through complex conjugate pairs by
$z_1\xi_0+z_1^*\xi_0^*+z_2\xi_1+z_2^*\xi_1^*$ where $z_1$, $z_2$ are complex numbers and 
$\xi_0(k_0)=\exp(ik_0x)\begin{pmatrix} 1\\ \frac{ \omega_c}{k_0}-ia_c k_0^3\end{pmatrix}$, $\xi_1 (k_0)=\exp(-ik_0x)\begin{pmatrix} 1\\-\frac{ \omega_c}{k_0}+ia_c k_0^3\end{pmatrix}$.
\end{proposition}

The above proposition can be concluded by entirely similar computations as in \cite{Yao}. However, we will do such computations on one hand to
make the later computations for getting the reduced dynamics in a solid foundation and on the other hand to track how $k_0$ enters the final reduced dynamics. The latter
is important and interesting for us.

\begin{proof}
For the wave number $k_0$ and eigenvalue $\lambda=i\omega_c$, the eigenfunctions $(\tau, u)$ satisfy the following system
\begin{equation}\label{V}
\begin{pmatrix}-a_c\partial_x^4 & \partial_x\\ \sigma'(0)\partial_x &
-\delta_c\partial_x^2-\partial_x^4\end{pmatrix}\begin{pmatrix}\tau\\ u
\end{pmatrix}=i\omega_c\begin{pmatrix}\tau\\ u
\end{pmatrix}.
\end{equation}
We seek solutions of the form $\exp(ik_0x)V=\exp(ik_0x)\begin{pmatrix} v_1\\ v_2 \end{pmatrix}$ and get the algebraic equation for the vector $V$ from \eqref{V} as follows
\begin{equation*}
\begin{pmatrix}-a_c k_0^4 &ik_0\\ \sigma'(0)ik_0 &
a_c k_0^4\end{pmatrix}\begin{pmatrix}v_1\\ v_2
\end{pmatrix}=i\omega_c\begin{pmatrix} v_1\\ v_2
\end{pmatrix},
\end{equation*}
i.e.,
\begin{equation}\label{V1}
\begin{pmatrix}-a_c k_0^4-i\omega_c &ik_0\\ \sigma'(0)ik_0 &
a_c k_0^4-i\omega_c\end{pmatrix}\begin{pmatrix}v_1\\ v_2
\end{pmatrix}=\begin{pmatrix} 0\\ 0
\end{pmatrix}.
\end{equation}
The coefficient matrix in \eqref{V1} is of rank one due to the definition of $\omega_c$ and the trivial fact that $ik_0\neq0$. Therefore, we get the form of  $V$ as 
$V=\begin{pmatrix} v_1\\( \frac{ \omega_c}{k_0}-ia_c k_0^3)v_1\end{pmatrix},$
which enables us to pick $\xi_0(k_0)$ as $\xi_0(k_0)=\exp(ik_0x)\begin{pmatrix} 1\\ \frac{ \omega_c}{k_0}-ia_c k_0^3\end{pmatrix}$

For the wave number $k_0$ and $\lambda=-i\omega_c$, the eigenfunctions $(\tau, u)$ satisfy the system
 \begin{equation}
\begin{pmatrix}-a_c\partial_x^4 & \partial_x\\ \sigma'(0)\partial_x &
-\delta_c\partial_x^2-\partial_x^4\end{pmatrix}\begin{pmatrix}\tau\\ u
\end{pmatrix}=-i\omega_c\begin{pmatrix}\tau\\ u
\end{pmatrix}.
\end{equation}
Seeking solutions of the form $\exp(ik_0x)V=\exp(ik_0x)\begin{pmatrix} v_1\\ v_2 \end{pmatrix}$, 
we get the algebraic system for the vector $V$ 
\begin{equation*}
\begin{pmatrix}-a_c k_0^4 &ik_0\\ \sigma'(0)ik_0 &
a_c k_0^4\end{pmatrix}\begin{pmatrix}v_1\\ v_2
\end{pmatrix}=-i\omega_c \begin{pmatrix} v_1\\ v_2
\end{pmatrix},
\end{equation*}
i.e.,
\begin{equation}\label{V2}
\begin{pmatrix}-a_c k_0^4+i\omega_c &ik_0\\ \sigma'(0)ik_0 &
a_c k_0^4+i\omega_c \end{pmatrix}\begin{pmatrix}v_1\\ v_2
\end{pmatrix}=\begin{pmatrix} 0\\ 0
\end{pmatrix}.
\end{equation}
Similarly, the coefficient matrix in \eqref{V2} is of rank one and we get the form of $V$ here as $V=\begin{pmatrix} v_1\\-( \frac{ \omega_c}{k_0}+ia_c k_0^3)v_1\end{pmatrix}$,
which enables us to pick an eigenfunction of $\mathcal L_c$ associated with $\lambda=-i\omega_c$ as $\xi_1(k_0)^*=\exp(ik_0x)\begin{pmatrix} 1\\-\frac{ \omega_c}{k_0}-ia_c k_0^3\end{pmatrix}$.

Next, we can write down the eigenfunctions of $\mathcal L_c$ for the wave number $-k_0$ with $\lambda=\pm i\omega_c$ by conjugacy or by repeating the above computations. Specifically, the eigenfunction for the wave number $-k_0$ associated with $\lambda=-i\omega_c$ is given by $\xi_0(k_0)^*=\exp(-ik_0x)\begin{pmatrix} 1\\ \frac{ \omega_c}{k_0}+ia_c k_0^3\end{pmatrix}$,
and the eigenfunction for the wave number $-k_0$ associated with $\lambda=i\omega_c$ is given by
$\xi_1 (k_0)=\exp(-ik_0x)\begin{pmatrix} 1\\-\frac{ \omega_c}{k_0}+ia_c k_0^3\end{pmatrix}$.
\end{proof}

We will adopt the specific parametrization of $Z_c$ in Proposition \ref{parametrization} in the remaining of the paper. Now we consider the group actions represented by $T_h, S$ which act on $U(x)$ through the spatial variable $x$ as follows:
$$T_h\begin{pmatrix} \tau(x)\\ u(x) \end{pmatrix}=\begin{pmatrix} \tau(x+h)\\ u(x+h) \end{pmatrix},\, h\in \mathbb{R}^1/[-\pi, \pi];S\begin{pmatrix} \tau(x)\\ u(x) \end{pmatrix}=\begin{pmatrix} \tau(-x)\\ -u(-x) \end{pmatrix}.$$

Our nonlinear perturbation system \eqref{bs} with parameters $\mu_1$ and $\mu_2$ has the form
\begin{equation}\label{equivariance}
\partial_t U(x,t)=\mathcal L_c U(x,t) +\begin{pmatrix} -\mu_1 \partial_x^4\tau(x,t)\\ -\mu_2\partial_x^2 u(x,t) \end{pmatrix}+\begin{pmatrix} 0\\ \partial_x f(\tau(x,t)) \end{pmatrix}
\end{equation}
where $f=f(\tau)$ is an arbitrary scalar function of $\tau$ only.
For system \eqref{equivariance}, we have the following proposition:

\begin{proposition}\label{proposition}
For any function $f=f(\tau)$, the system \eqref{equivariance} on $\mathbb{R}^1/[-\pi, \pi]$ 
exhibits $O(2)$-symmetry, which is represented by $[T_h, \mathcal L_c]=0,\,\,[T_h, \mathcal R]=0, \,\, [S, \mathcal R]=0,\,\, [S, \mathcal L_c]=0,\,\, T_hS=ST_{-h}$, 
where $\mathcal R(U):=\begin{pmatrix} -\mu_1\partial_x^4\\ -\mu_2\partial_x^2 u(x,t) \end{pmatrix}+\begin{pmatrix} 0\\ \partial_x f(\tau(x,t)) \end{pmatrix}$.
\end{proposition}

\begin{proof}
The proof is simple symbolic computations. The tricky point is that we also need to put negative signs in the derivatives when the $S$ action is involved. Due to the importance of the symmetry property, we verify the conclusions here.  

(1) We first verify that $[T_h, \mathcal L_c]=0$. 
\begin{align*}
T_h\mathcal L_c U(x)&=T_h\begin{pmatrix} -a_c\partial_x^4\tau(x)+\partial_xu(x)\\ \sigma'(0)\partial_x\tau(x)-\delta_c\partial_x^2u(x)-\partial_x^4 u(x) \end{pmatrix}\\
&=\begin{pmatrix} -a_c\partial_{x+h}^4\tau(x+h)+\partial_{x+h}u(x+h)\\ \sigma'(0)\partial_{x+h}\tau(x+h)-\delta_c\partial_{x+h}^2u(x+h)-\partial_{x+h}^4 u(x+h) \end{pmatrix}\\
&=\begin{pmatrix} -a_c\partial_{x}^4\tau(x+h)+\partial_{x}u(x+h)\\ \sigma'(0)\partial_{x}\tau(x+h)-\delta_c\partial_{x}^2u(x+h)-\partial_{x}^4 u(x+h) \end{pmatrix}.
\end{align*}
On the other hand, we have
\begin{align*}
\mathcal L_c T_hU(x)&=\begin{pmatrix} -a_c\partial_x^4 & \partial_x\\ \sigma'(0)\partial_x & -\delta_c\partial_x^2u(x)-\partial_x^4 u(x) \end{pmatrix}\begin{pmatrix}\tau(x+h)\\ u(x+h)\end{pmatrix}\\
&=\begin{pmatrix} -a_c\partial_{x}^4\tau(x+h)+\partial_{x}u(x+h)\\ \sigma'(0)\partial_{x}\tau(x+h)-\delta_c\partial_{x}^2u(x+h)-\partial_{x}^4 u(x+h) \end{pmatrix}.
\end{align*}
Hence we obtain $[T_h, \mathcal L_c]=0$.

(2) $[T_h, \mathcal R]=0$ can be verified similarly as in (1).

(3) Now we verify that $[S, \mathcal R]=0$. 
We have 
\begin{align*}
S\mathcal R(U)&=S\begin{pmatrix} -\mu_1\partial_x^4\tau(x,t)\\ -\mu_2\partial_x^2 u(x,t) \end{pmatrix}+S\begin{pmatrix} 0\\ \partial_x f(\tau(x,t)) \end{pmatrix}\\
&=\begin{pmatrix} -\mu_1\partial_{-x}^4\tau(-x, t)\\ -\Big(-\mu_2\partial_{-x}^2 u(-x,t)\Big) \end{pmatrix}+\begin{pmatrix} 0\\ -\partial_{-x} f(\tau(-x,t)) \end{pmatrix}\\
&=\begin{pmatrix}  -\mu_1\partial_{x}^4\tau(-x, t)\\ \mu_2\partial_{x}^2 u(-x,t)\Big) \end{pmatrix}+\begin{pmatrix} 0\\ \partial_{x} f(\tau(-x,t)) \end{pmatrix}.
\end{align*}
Noticing that
\begin{align*}
\mathcal R(SU)&=\begin{pmatrix} -\mu_1\partial_x^4\tau(-x,t)\\ -\mu_2\partial_x^2 (-u(-x,t)) \end{pmatrix}+\begin{pmatrix} 0\\ \partial_x f(\tau(-x,t)) \end{pmatrix}\\
&=\begin{pmatrix} -\mu_1\partial_x^4\tau(-x,t) \\ \mu_2\partial_{x}^2 u(-x,t)\Big) \end{pmatrix}+\begin{pmatrix} 0\\ \partial_{x} f(\tau(-x,t)) \end{pmatrix},
\end{align*}
we see $[S, \mathcal R]=0$ by comparing the two expressions.

(4) $[S, \mathcal L_c]=0$ can be verified similarly as in (3).

(5) Finally, we verify that $T_hS=ST_{-h}$. 
It is easy to see that
$$T_hSU=T_h\begin{pmatrix}\tau(-x)\\-u(-x) \end{pmatrix}=\begin{pmatrix}\tau(-(x+h))\\-u(-(x+h)) \end{pmatrix},$$
$$ST_{-h}U=S\begin{pmatrix}\tau(x-h)\\u(x-h) \end{pmatrix}=\begin{pmatrix}\tau(-x-h)\\-u(-x-h) \end{pmatrix}.$$
Hence we have $T_hS=ST_{-h}$.
\end{proof}

Coming back to our nonlinear perturbation system \eqref{perturb} or bifurcation system  
\eqref{bifurcation equation}, we have the following corollary:

\begin{corollary}\label{coro}
Both systems \eqref{perturb} and \eqref{bifurcation equation}  exhibit $O(2)$ symmetry. Also, the system given by
$\partial_t U=\mathcal L_c U+R_{11}(U, \mu_1,\mu_2)+R_{20}(U, U)+R_{30}(U, U, U)$
on $\mathbb{R}^1/[-\pi, \pi]$ also exhibits $O(2)$ symmetry.
\end{corollary}

We emphasize that it is Proposition \ref{proposition} or Corollary \ref{coro} that enables us to extend our former work in \cite{Yao} to
the case of two bifurcation parameters and general flux functions. More precisely, the symmetry property of the system is preserved even when we
introduce one more parameter which enters through a term $-\mu_1\partial_x^4\tau(x,t)$ and add or drop higher order terms in the expansion of $\sigma(\tau)$.
The symmetry helps us  both in identifying the normal form of the dynamics as in \cite{HI} and in concluding the bifurcation dynamics. As we need to analyze a reduced dynamics for the latter purpose, it is crucial for us to avoid such a situation that the system admits the desired symmetry after dropping higher order terms but does not before and \textit{vice versa}. 

To do effective computations later, we need the representation of the operator $\mathcal L_c$ on the center space and a duality argument. 

For the former, we collect some information of the actions of $\mathcal L_c$, $T_h$ and $S$ on the center space $Z_c$. Under our chosen parametrization of the center space $Z_c$, $\mathcal L_c$ is represented on $Z_c$ by the diagonal matrix $diag\{i\omega_c,-i\omega_c,i\omega_c,-i\omega_c\}$. The actions of $T_h$ and $S$ on the complex ``basis" $\{\xi_0(k_0), \xi_0(k_0)^*,  \xi_1(k_0), \xi_1(k_0)^*\}$ are specified  by $S\xi_0(k_0)=\xi_1(k_0)$, $S\xi_1(k_0)=\xi_0(k_0)$, $S\xi_0(k_0)^*=\xi_1(k_0)^*$, $S\xi_1(k_0)^*=\xi_0(k_0)^*$, $T_h\xi_0(k_0)=\exp(ih)\xi_0(k_0)$, $T_h\xi_1(k_0)=\exp(-ih)\xi_1(k_0)$, $T_{-h}\xi_0(k_0)^*=\exp(-ih)\xi_0(k_0)^*$, $T_{-h}\xi_1(k_0)^*=\exp(ih)\xi_1(k_0)^*$. The above expressions can be easily verified in view of the definitions of $\xi_0(k_0)$, $\xi_1(k_0)$,
$T_h$ and $S$.

For the latter, we define $\eta=\eta(k_0)$ for the later use of duality as follows:
\begin{equation}\label{eta}
\eta(k_0)=\frac{k_0}{4\pi \omega_c}\exp(ik_0x)\begin{pmatrix} -ia_c k_0^3+\frac{\omega_c}{k_0}\\1\end{pmatrix}
\end{equation}
It is easy to verify that 
$$\langle \xi_0(k_0), \eta(k_0)\rangle=\int\xi_0(k_0)\cdot \eta(k_0)^*\,dx=1.$$
Actually, taking into account the two relations $\delta_c=(a_c+1)k_0^2$ and $\omega_c^2=\sigma'(0)k_0^2-a_c^2k_0^8$, one can easily check that
\begin{equation}
(i\omega_c-\mathcal L_c)^*\eta=\begin{pmatrix}-i\omega_c+a_c\partial_x^4 & \sigma'(0)\partial_x \\ \partial_x & -i\omega_c+\delta_c\partial_x^2+\partial_x^4 \end{pmatrix}\eta=0,
\end{equation}
where $(i\omega_c-\mathcal L_c)^*$ is the adjoint of the linear operator $i\omega_c-\mathcal L_c$.

Now we consider the symmetry and tangency in the center manifold theory and normal form theory.
The guiding principle is that the center manifold reduction function and the normal form transformation function inherit the symmetries from the
partial differential equations. In other words, the center manifold reduction function and the normal form transformation function are invariant under the same group actions if the partial differential equations are invariant under some group actions. To fix the ideas and make the exposition clear, let us first cite a version of the parameter dependent center manifold theorem with group actions and a version of normal form theorem with group actions.

\begin{theorem} (Parameter dependent center manifold theorem with symmetries, see \cite{HI, GSS, Yao})\label{center manifold theorem}
Let the inclusions in the Banach space triplet $\mathcal Z \subset \mathcal Y\subset \mathcal X$ be continuous. Consider a differential equation in a Banach space $\mathcal X$ of the form
$$\frac{du}{dt}=\mathcal{L}u+\mathcal{R}(u,\mu)$$ and assume that

(1) (Assumption on linear operator and nonlinearity) $\mathcal{L}:\mathcal Z\mapsto\mathcal X$ is a bounded linear map
and for some $k\geq 2$, there exist neighborhoods
$\mathscr{V_u}\subset\mathcal Z$ and
$\mathscr{V_\mu}\subset\mathbb{R}^m$ of $(0,0)$ such that
$\mathcal{R}\in C^k(\mathscr{V}_u\times \mathscr{V_{\mu}},
\mathcal Y)$ and
$$\mathcal{R}(0,0)=0,\,\, D_u\mathcal{R}(0,0)=0.$$

(2) (Spectral decomposition) there exists some constant $\gamma>0$ such that
$$\inf\{Re\lambda; \lambda\in \sigma_u(\mathcal L)\}>\gamma,\,\, \sup\{Re\lambda; \lambda\in \sigma_s(\mathcal L)\}<-\gamma,$$
and the set $\sigma_c(\mathcal L)$ consists of a finite number of eigenvalues with finite algebraic multiplicities.

(3) (Resolvent estimates) For Hilbert space triplet $\mathcal Z \subset \mathcal Y\subset \mathcal X$, assume
there exists a positive constant $\omega_0>0$ such that $i\omega\in \rho(\mathcal L)$ for all $|\omega|>\omega_0$ and
$\|(i\omega-\mathcal L)^{-1}\|_{\mathcal X\mapsto \mathcal X}\lesssim\frac{1}{|\omega|}$. For Banach space triplet, we need further
$\|(i\omega-\mathcal L)^{-1}\|_{\mathcal Y\mapsto \mathcal X}\lesssim\frac{1}{|\omega|^{\alpha}}$ for some $\alpha\in [0,1)$.

Then there exists a map $\Psi\in\mathcal C^k(Z_c, Z_h)$ and a neighborhood 
$\mathscr O_{u}\times \mathscr O_{\mu}$of $(0, 0)$ in $\mathcal Z\times \mathbb R^m$ such that

(a) (Tangency) $\Psi(0, 0)=0$ and $D_u\Psi(0,0)=0$.

(b) (Local flow invariance) the manifold $\mathcal M_0 (\mu)=\{ u_0+\Psi(u_0,\mu); u_0\in Z_c \}$ has the properties.
(i) $\mathcal M_0 (\mu)$ is locally invariant, i.e., if $u$ is a solution satisfying $u(0)\in \mathcal M_0 (\mu)\cap \mathscr O_{\mu}$ and $u(t)\in\mathcal O_{u}$ for all $t\in [0, T]$, then
$u(t)\in\mathcal M_0 (\mu)$ for all $t\in [0, T]$; (ii) $\mathcal M_0 (\mu)$ contains the set of bounded solutions staying in $\mathcal O_{u}$ for all $t\in\mathbb{R}^1$, i.e., if $u$ is a solution satisfying
$u(t)\in\mathcal O_u$ for all $t\in\mathbb R^1$, then $u(0)\in \mathcal M_0 (\mu)$.

(c) (Symmetry) Moreover, if the vector field is equivariant in the sense that there
exists an isometry $\mathscr{T}\in \mathcal L(\mathcal X)\cap \mathcal L(\mathcal Z)$
which commutes with the vector field in the original system,
$$[\mathscr{T}, \mathcal{L}]=0,\,\,[\mathscr{T}, \mathcal{R}]=0,$$
then the $\Psi$ commutes
with $\mathscr{T}$ on $Z_c$: $[\Psi, T]=0$.
\end{theorem}

\begin{theorem} (Parameter dependent normal form theorem with symmetries \cite{HI, GSS, Yao})\label{normal form theorem}
Consider a differential equation in $\mathbb{R}^n$ of the form
$$\frac{du}{dt}=\mathcal{L}u+\mathcal{R}(u,\mu)$$ and assume that

(1) $\mathcal{L}$ is a linear map in $\mathbb{R}^n$;\\

(2) for some $k\geq 2$, there exist neighborhoods
$\mathscr{V_u}\subset\mathbb{R}^n$ and
$\mathscr{V_\mu}\subset\mathbb{R}^m$ of $(0,0)$ such that
$\mathcal{R}\in C^k(\mathscr{V}_u\times \mathscr{V_{\mu}},
\mathbb{R}^n)$ and
$$\mathcal{R}(0,0)=0,\,\, D_u\mathcal{R}(0,0)=0.$$
Then for any positive integer p, $k>p\geq2$, there exists
neighborhoods $\mathscr{V}_1$ and $\mathscr{V}_2$ of $(0,0)$ in
$\mathbb{R}^n\times \mathbb{R}^m$ such that for any $\mu\in
\mathscr{V}_2$, there is a polynomial
$\Pi_{\mu}:\mathbb{R}^n\rightarrow\mathbb{R}^n$ of degree $p$ with
the following properties.

(i) The coefficients of the monomials of degree $q$ in $\Pi_{\mu}$
are functions of $\mu$ of class $C^{k-q}$, and
$$\Pi_0(0)=0,\,\, D_u \Pi_0(0)=0.$$

(ii) For $v\in\mathscr{V}_1$, the polynomial change of variables
$$u=v+\Pi_{\mu}(v)$$
transforms the original system into the normal form
\begin{equation*}
\frac{dv}{dt}=\mathcal{L}v+\mathscr{N}_{\mu}(v)+\rho(v,\mu),
\end{equation*} such that

(a) (Tangency) For any $\mu\in \mathscr{V}_2$, $\mathscr{N}_{\mu}$ is a
polynomial $\mathbb{R}^n\rightarrow\mathbb{R}^n$ of degree $p$, with
coefficients depending on $\mu$, such that the coefficients of the
monomials of degree $q$ are of class $C^{k-q}$, and
$$\mathscr{N}_0(0)=0,\,\, D_v\mathscr{N}_0(0)=0.$$

(b) (Characteristic condition) The equality
$$\mathcal{N}_{\mu}(e^{t\mathcal{L}^*}v)=e^{t\mathcal{L}^*}\mathcal{N}_{\mu}(v)$$
holds for all $(t,\mu)\in\mathbb{R}\times\mathbb{R}^n$ and
$\mu\in\mathscr{V}_2$.

(c) (Smoothness )The map $\rho$ belongs to
$C^k(\mathscr{V}_1\times\mathscr{V}_2,\mathbb{R}^n)$, and
$$\rho(v,\mu)=o(|v|^p)$$
for all $\mu\in \mathscr{V}_2$.

(d) (Symmetry) Moreover, if the vector field is equivariant in the sense that there
exists an isometry $\mathscr{T}:\mathbb{R}^n\rightarrow\mathbb{R}^n$
which commutes with the vector field in the original system,
$$[\mathscr{T}, \mathcal{L}]=0,\,\,[\mathscr{T}, \mathcal{R}]=0,$$
then the polynomials $\Pi_{\mu}$ and $\mathscr{N}_{\mu}$ commute
with $\mathscr{T}$ for all $\mu\in\mathscr{V}_2$.
\end{theorem}

We have the following theorem regarding system \eqref{bifurcation equation}.

\begin{theorem}\label{plemma}
(Existence of parameter-dependent center manifolds with symmetry) For system \eqref{bifurcation equation}, there exists a map
$\Psi\in\mathcal{C}^k( Z_c\times \mathbb{R}^2, Z_h)$, with
$$\Psi(0, 0, 0)=0,\,\, D_U\psi(0,0, 0)=0,$$
and a neighborhood of $\mathcal{O}_U\times\mathcal{O}_{(\mu_1, \mu_2)}$ of $(0,
0)$ such that for $(\mu_1,\mu_2)\in \mathcal{O}{(\mu_1,\mu_2)}$, the manifold
$$\mathcal{M}_0(\mu_1,\mu_2):=\{U_0 +\Psi(U_0, \mu_1,\mu_2);\,\, U_0\in Z_c\}$$
is locally invariant and contains the set of bounded solutions of
the nonlinear perturbation system in $\mathcal{O}_U$ for all
$t\in\mathbb{R}$. Moreover, the center manifold reduction function is $O(2)$ equivariant on $ Z_c$.
\end{theorem}

The proof of this theorem is essentially the same as the one in \cite{Yao}. However, we need to reproduce it here to clarify the ambiguity due to misprints/typos in \cite{Yao} which make the 
proof there hard to understand, and for completeness.

\begin{proof}
From the definition of the operator $\mathcal L_c$ and $R(U,\mu_1, \mu_2)$ in \eqref{bifurcation equation}, we know that the 
assumptions $(1)$ and $(2)$ in Theorem \ref{center manifold theorem} on the linear operator and nonlinearity hold. 
A subtle point here is to verify that $R(U,\mu)$ is in $Y$. In fact, if $\mu_1=0$, we can choose $Y$ to be the subspace of $H^2_{per}(-\pi, \pi)$ with elements having mean zero.
The spectrum decomposition assumption is a direction consequence of Proposition \ref{cp} and Lemma \ref{gap}. It is obvious that $i\omega\in\rho(\mathcal L_c)$ for $|\omega|\neq\omega_c$. 
To show the resolvent estimate in the Hilbert space triplet setting, we write $(i\omega-\mathcal L_c) U=\tilde U$ for $\tilde U=\begin{pmatrix}\tilde\tau\\ \tilde u \end{pmatrix} \in X$ and 
$U=\begin{pmatrix}\tau \\ u \end{pmatrix} \in Z $ and show that $\|U\|_{X}\lesssim\frac{1}{|\omega|}\|\tilde U\|_{X}$ for $|\omega|>\omega_0>\omega_c$ where $\omega_0\gg 1$ is a large constant. Without loss of generality, we just need to prove for $\omega\geq \omega_0$. Multiplying the equation $(i\omega-\mathcal L_c) U=\tilde U$ by $U^*$, integrating over $[-\pi, +\pi]$ and integrating by parts, we arrive at
$$i\omega|\tau|^2_{L^2}+a_c|\partial_x^2\tau|_{L^2}^2-\int \partial_x u\tau^*=\int \tilde{\tau}\tau^*;$$
$$\int-\sigma'(0)\partial_x\tau u^*+i \omega |u|_{L^2}^2-(a_c+1)k_0^2\int|\partial_x u|^2+\int|\partial_x^2 u|^2=\int\tilde{u}u^*.$$
Taking the imaginary and real parts respectively, we see

$$\omega|\tau|_{L^2}^2=Im\int \partial_x u\tau^* +Im \int \tilde{\tau}\tau^*,$$
$$\omega |u|_{L^2}^2=  \sigma'(0) Im  \int \partial_x\tau u^* +Im \int \tilde{u}u^*$$
and
$$a_c |\partial_x^2\tau|^2_{L^2}=Re\int \partial_x u\tau^* +Re\int \tilde{\tau}\tau^*,$$
$$-(a_c+1)k_0^2 |\partial_x u|^2_{L^2}+\int|\partial_x^2 u|^2=Re\int\tilde{u}u^* +Re\int \sigma'(0)\partial_x\tau u^*.$$

From the imaginary part equations, we get by using elementary
inequalities, Fourier analysis and mean zero property for elements in the space $X$, that
\begin{equation}\label{r}
\omega |\tau|^2_{L^2}\lesssim\frac{1}{\omega}\Big(|\partial_x u|^2_{L^2}+|\tilde{\tau}|^2_{L^2}  \Big)\leq \frac{1}{\omega}\Big(|\partial^2_x u|^2_{L^2}+|\tilde{\tau}|^2_{L^2}  \Big),
\end{equation}
\begin{equation}\label{i}
\omega |u|_{L^2}^2\lesssim \frac{1}{\omega}\Big(|\partial_x\tau|^2_{L^2}+|\tilde{u}|^2_{L^2}  \Big)\leq  \frac{1}{\omega}\Big(|\partial^2_x\tau|^2_{L^2}+|\tilde{u}|^2_{L^2}  \Big).
\end{equation}
From the real part equations, we get for $a_c\neq 0$ that
\begin{equation}\label{real1}
|\partial_x^2\tau|^2_{L^2}\lesssim|\tau|_{L^2}\Big( |\tilde{\tau}|_{L^2}+|\partial_x u|_{L^2}  \Big),\end{equation}
\begin{equation}\label{real2}
|\partial_x^2 u|^2_{L^2}\lesssim|u|_{L^2}\Big(
|\tilde{u}|_{L^2}+|\partial_x\tau|_{L^2}  \Big)+|\partial_x
u|^2_{L^2}.
\end{equation}
By interpolation, we know that for any $\epsilon>0$, the following holds
$$|\partial_x u|^2_{L^2}\leq \epsilon|\partial_x^2 u|^2_{L^2}+C(\epsilon)|u|^2_{L^2}.$$
We may pick $\epsilon$ so small that we can conclude from $\eqref{real2}$
that
$$|\partial_x^2 u|^2_{L^2}\lesssim|u|_{L^2}\Big(|\tilde{u}|_{L^2}+|\partial_x\tau|_{L^2} +|u|_{L^2}  \Big).$$
By interpolation in $|\partial_x\tau|_{L^2}$,
$$|\partial_x\tau|^2_{L^2}\lesssim\epsilon |\partial_x^2\tau|^2_{L^2}+C(\epsilon)|\tau|^2_{L^2},$$
we can get from \eqref{real1} and \eqref{real2} that
\begin{equation}\label{3}
\begin{cases}
|\partial_x^2\tau|^2_{L^2}\lesssim
|\tau|^2_{L^2}+|\tilde{\tau}|^2_{L^2}+\epsilon|\partial_x^2 u|^2_{L^2}+|u|^2_{L^2},\\
|\partial_x^2 u|^2_{L^2}\lesssim
|u|^2_{L^2}+|\tilde{u}|^2_{L^2}+\epsilon|\partial_x^2\tau|^2_{L^2}+\epsilon|\partial_x^2 u|^2_{L^2}+|\tau|^2_{L^2}.
\end{cases}
\end{equation}
Adding the two equations in \eqref{3} together and choosing $\epsilon$
smaller if necessary, we get
\begin{equation}\label{4}
|\partial_x^2\tau|^2_{L^2}+|\partial_x^2u|^2_{L^2}\lesssim|\tau|^2_{L^2}+|u|^2_{L^2}+|\tilde{\tau}|^2_{L^2}+|\tilde{u}|^2_{L^2}.
\end{equation}
Now we have in view of \eqref{r}, \eqref{i} and \eqref{4} that
\begin{align*}
\omega|\tau|^2_{L^2}&\lesssim\frac{1}{\omega}|\tilde{\tau}|^2_{L^2}+\frac{1}{\omega}|\partial_x^2
u|^2_{L^2}\\
&\lesssim\frac{1}{\omega}|\tilde{\tau}|^2_{L^2}+\frac{1}{\omega}\Big(|\tau|^2_{L^2}+|u|^2_{L^2}+|\tilde{\tau}|^2_{L^2}+|\tilde{u}|^2_{L^2}
\Big).
\end{align*}
\begin{align*}
\omega |u|^2_{L^2}&\lesssim\frac{1}{\omega}|\tilde{u}|^2_{L^2}+\frac{1}{\omega}|\partial_x^2
\tau|^2_{L^2}\\
&\lesssim\frac{1}{\omega}|\tilde{u}|^2_{L^2}+\frac{1}{\omega}\Big(|\tau|^2_{L^2}+|u|^2_{L^2}+|\tilde{\tau}|^2_{L^2}+|\tilde{u}|^2_{L^2}
\Big).
\end{align*}
Adding the above two inequalities together, we get
$$(-\frac{1}{\omega}+\omega) |U|^2_{L^2}\lesssim \frac{1}{\omega} |\tilde U|^2_{L^2},$$
which implies
$$|U|^2_{L^2}\lesssim\frac{1}{\omega(\omega-\frac{1}{\omega})}|\tilde U|^2_{L^2}
=\frac{1}{\omega^2 -1}|\tilde U|^2_{L^2}
\lesssim\frac{1}{\omega^2}|\tilde U|^2_{L^2}.$$
for $\omega\geq \omega_0$ large. Hence we arrive at the estimate of desired
form $|U|_{L^2}\lesssim\frac{1}{\tilde{\omega}}|\tilde U|_{L^2}$ if $a_c\neq 0$. 
If $a_c=0$, the terms $|\partial^2_x\tau|^2_{L^2}$ will not get involved in the estimates and the proof of resolvent estimate is significantly simpler.

From Corollary \eqref{coro}, system \eqref{bifurcation equation} exhibits $O(2)$ symmetry, so does the function $\Psi$.
\end{proof}

\begin{remark}\label{gcm}
The main concern of the above theorem is the behavior of the resolvent of $\mathcal L_c$ along the imaginary axis and far from the origin. The validity of the theorem lies in both the structure of the linear operator $\mathcal L_c$ itself and the working spaces under which we choose to consider the spectral problem. In general, the inequalities in \eqref{r} and \eqref{i} would not be true if $\int_{-\pi}^{\pi}\tau\,dx \neq 0$ and $\int_{-\pi}^{\pi} u\,dx\neq 0$. Hence the existence of center manifold without mean zero assumption in the space triplet $Z\subset Y\subset X$ would be a problem here. 
\end{remark}

\begin{remark}
The expression $(4.8)$ in page 67 of Yao \cite{Yao} should be replaced by
\begin{equation*}
\begin{cases}
|\partial_x^2\tau|^2_{L^2}\lesssim
|\tau|^2_{L^2}+|\tilde{\tau}|^2_{L^2}+\epsilon|\partial_x^2 u|^2_{L^2}+|u|^2_{L^2},\\
|\partial_x^2 u|^2_{L^2}\lesssim
|u|^2_{L^2}+|\tilde{u}|^2_{L^2}+\epsilon|\partial_x^2\tau|^2_{L^2}+\epsilon|\partial_x^2 u|^2_{L^2}+|\tau|^2_{L^2}.
\end{cases}
\end{equation*}
\end{remark}

Let us briefly discuss the tangency and decomposition. The tangency properties in the center manifold theorem (Theorem \ref{center manifold theorem}) and normal form theorem (Theorem \ref{normal form theorem}) are crucial for determining the reduced dynamics. To get the reduced dynamics, the routine way is to decompose $U\in Z_c$ as $U=U_c+\Psi(U_c, \mu)$ where $U_c\in Z_c$ and $\Psi(U_c, \mu)\in Z_h$ and then use the flow invariance. The following observation will help us in identifying the reduced dynamics. The reduced dynamics for $U_c$ is a finite dimensional system. In view of the normal form theorem, we can decompose $U_c$ at the beginning as $U_c=V_c +\Pi_{\mu}(V_c)$ where $U_c\in Z_c$ and $\Pi_{\mu}(V_c)\in Z_c$ according to normal form theorem. As a result, we get $U=V_c+\Pi_{\mu}(V_c)+\Psi(V_c +\Pi_{\mu}(V_c), \mu):=V_c+\Psi(V_c, \mu)$ with $V_c\in Z_c$ and $\Psi(V_c, \mu):=\Pi_{\mu}(V_c)+\Psi(V_c +\Pi_{\mu}(V_c), \mu)\in Z$. 
We find  that the tangency is preserved:
$\Psi(V_c,\mu)\Big |_{V_c=0,\mu=0}=0$ and $D_{V_c}\Psi(V_c,\mu)\Big |_{V_c=0,\mu=0}=0$.
On the center space $Z_c$, we can coordinate $V_c$ as $V_c=z_1(t)\xi_0(k_0)+z_1(t)^*\xi_0(k_0)^*+z_2(t)\xi_1(k_0)+z_2(t)\xi_1(k_0)^*$. Now 
\begin{align*}
U&=z_1(t)\xi_0(k_0)+z_1(t)^*\xi_0(k_0)^*+z_2(t)\xi_1(k_0)+z_2(t)^*\xi_1(k_0)^*\\&+\Psi(z_1 (t)\xi_0(k_0)+z_1 (t)^*\xi_0(k_0)^*+z_2(t)\xi_1(k_0)+z_2(t)^*\xi_1(k_0)^*, \mu)\\&:=z_1(t)\xi_0(k_0)+z_1(t)^*\xi_0(k_0)^*+z_2(t)\xi_1(k_0)+z_2(t)^*\xi_1(k_0)^*\\&+\Psi(z_1(t), z_1(t)^*, z_2(t), z_2(t)^*).
\end{align*} In determining the parameters in the $O(2)$ equivariant normal form in the following section, we will adopt this decomposition.

\section{Analysis}\label{analysis}

\setcounter{equation}{0}
\setcounter{theorem}{0}
Though a normal form is not always necessary in computing the dynamics on the center manifold(s) as in certain cases one can directly construct approximations of the center manifold function(s) by flow invariance, it helps a lot in making computations. In our computations here, we will use both approximation and normal form.
From the spectral scenario of $\mathcal L_c$, the characteristic condition $(b)$ on the possible form of normal form and the symmetry of the nonlinear 
perturbation system \eqref{perturb}, we could determine a normal form of the reduced dynamics. The analysis we have done to now and the analysis in
\cite{HI} (see pages 129-131) yield the following theorem:

\begin{theorem}\label{prop}
The reduced dynamical system of \eqref{bifurcation equation} on the parameter dependent center manifold has
the following normal form
\begin{equation}\label{normalform}
\begin{cases}
\frac{d}{dt}z_1=i\omega_c z_1+ z_1P(|z_1|^2, |z_2|^2, \mu_1, \mu_2)+\rho(z_1, z_2, z_1^*, z_2^*,\mu_1,\mu_2)\\
\frac{d}{dt}z_2=i\omega_c z_2+ z_2P(|z_2|^2, |z_1|^2, \mu_1, \mu_2)+\rho(z_2, z_1, z_2^*, z_1^*,\mu_1,\mu_2)
\end{cases}
\end{equation}
in which $P$ is a polynomial of degree $p$ in its first two
arguments with coefficients depending on $\mu_1$ and $\mu_2$, and $\rho(z_1, z_2, z_1^*,
z_2^*, \mu_1, \mu_2)=\mathcal{O}((|z_1|+|z_2|)^{2p+3})$ with the following property, for $h\in\mathbb{R}^1/[-\pi, \pi]$,
$$\rho\Big(\exp(ih)z_1, \exp(ih)z_2, \exp(-ih)z_1^*,
\exp(-ih)z_2^*, \mu_1, \mu_2\Big)=\exp(ih)\rho(z_1, z_2, z_1^*, z_2^*, \mu_1,\mu_2).$$ Therefore, the normal form takes the following form
\begin{equation}\label{nf}
\begin{cases}
\frac{d}{dt}z_1=i\omega_c z_1+ z_1(\mathbbm{a}(\mu_1,\mu_2)+\mathbbm{b}(\mu_1,\mu_2)|z_1|^2+\mathbbm{c}(\mu_1,\mu_2)|z_2|^2)+\cdot\cdot\cdot\\
\frac{d}{dt}z_2=i\omega_c z_2+
z_2(\mathbbm{a}(\mu_1,\mu_2)+\mathbbm{b}(\mu_1,\mu_2)|z_2|^2+\mathbbm{c}(\mu_1,\mu_2)|z_1|^2)+\cdot\cdot\cdot
\end{cases}
\end{equation}
where $\mathbbm{a,b,c}$ are complex coefficients depending on $\mu_1$ and $\mu_2$.
\end{theorem}

Notice that the analysis of \cite{HI} (see pages 129-131) in determining a normal form is for the case $\mu\in \mathbbm R^1$. But the analysis there holds for multidimensional parameter by checking
the statement of the normal form theorem (Theorem \ref{normal form theorem}) and repeating the argument. Unfortunately, the normal form in page 132 of \cite{HI} does not work here due to the higher dimension of the bifurcation parameters. In our case here, it is much more complicated. After realizing these issues, we now proceed to examine the coefficients $\mathbbm{a, b, c}$ in the above normal form. Our strategies to compute these coefficients are the following: (1) we first compute $\mathbbm {a}$ by the first order approximation of the center manifold function. (2) we examine $\mathbbm{b,c}$ after obtaining $\mathbbm{a}$. For our purpose, we do not need to determine $\mathbbm{b, c}$ completely. During the complicated process, we need also come back and forth to get needed ingredients. The reason that we can not avoid the back-and-forth procedure is the dependence of $\mathbbm{b,c}$ on $\mu_1,\mu_2$

To proceed, we write $\Psi(z_1, z_2, z_1^*, z_2^*, \mu)$ as 
$$\Psi(z_1, z_2, z_1^*, z_2^*, \mu)=\sum_{p+q+r+s+l+m\geq1}\Psi_{pqrslm}z_1^pz_1^{*q}z_2^rz_2^{*s}\mu_1^l\mu_2^m,$$
By conjugacy and the fact that our original system is real, we have
$$\Psi_{pqrslm}=\Psi^*_{qpsrlm}\in Y.$$
By tangency in Theorem \ref{center manifold theorem}, Theorem \ref{normal form theorem} and the discussion at the end of Section \ref{symmetry}, we have
$$ \Psi_{100000}=\Psi_{010000}=\Psi_{001000}=\Psi_{000100}=0.$$
Therefore, we can write the function $\Psi(z_1, z_2, z_1^*, z_2^*, \mu)$ as a sum with summands given by
$$\Psi_{000010}\mu_1, \Psi_{000001}\mu_2, \Psi_{100010}z_1\mu_1, \Psi_{010010}z_1^*\mu_1, \Psi_{10001}z_1\mu_2, \Psi_{01001}z_1^*\mu_2,\Psi_{001010}z_2\mu_1,$$ 
$$\Psi_{000110}z_2^*\mu_1, \Psi_{000101}z_2\mu_2,\Psi_{000101}z_2^*\mu_2, \Psi_{110000}z_1z_1^*, \Psi_{200000}z_1^2, \Psi_{020000}z_1^{*2},\Psi_{001100}z_2z_2^*,$$
 $$ \Psi_{002000}z_2^2, \Psi_{000200}z_2^{*2}, \Psi_{101000}z_1z_2, \Psi_{010100}z_1^*z_2^*, \Psi_{100100}z_1z_2^*, \Psi_{011000}z_1^*z_2, \Psi_{111000}z_1z_1^*z_2,$$
 $$\Psi_{110100}z_1z_1^*z_2^*, \Psi_{101100}z_1z_2z_2^*, \Psi_{011100}z_1^*z_2z_2^*, \Psi_{21000}z_1^2z_1^*, etc.$$
Next, we proceed the computations which are very complicated. However, there are also some good propositions hidden in the structure which can be seen from the coefficients in front of
the terms $R_{20}$ and $R_{30}$ below.

\textbf{step1}. Compute $\mathbbm{a}$. Though the center manifold reduction functions are not unique, their approximations up to any finite order allowed by the
smoothness of the nonlinearity are unique once a basis of the center space is chosen. By our choice of the basis for the center manifold, the first approximation of the dynamics on
the center manifold is given by
$$\partial_t(z_1\xi_0+z_1^*\xi_0^*+z_2\xi_1+z_2^*\xi_1^*)=L_c (z_1\xi_0+z_1^*\xi_0^*+z_2\xi_1+z_2^*\xi_1^*)+ R_{11}(z_1\xi_0+z_1^*\xi_0^*+z_2\xi_1+z_2^*\xi_1^*)+h.o.t.$$  
where ``$h.o.t$" means ``higher order terms". Let $P$ denote the canonical projection onto the center space. The above equation can be written as
\begin{equation}
\frac{d}{dt}\begin{pmatrix} z_1\\ z_1^*\\ z_2\\ z_2^*\end{pmatrix}=\begin{pmatrix} i\omega_c & 0 &0 &0 \\ 0 & -i\omega_c & 0 &0 \\ 0& 0& i\omega_c &0\\ 0& 0 & 0 & -i\omega_c\end{pmatrix}\begin{pmatrix} z_1\\ z_1^*\\ z_2\\ z_2^*\end{pmatrix}+PR_{11}(z_1\xi_0+z_1^*\xi_0^*+z_2\xi_1+z_2^*\xi_1^*)+h.o.t.
\end{equation}
To determine $\mathbbm{a}$, we do not need the full projection $P$. Rather, we will select one equation, say the equation on $z_1$. Comparing with the normal form, we see that the coefficient $\mathbbm a$ is given by
\begin{equation}\label{aaa}
\mathbbm az_1=\langle R_{11}(z_1\xi_0+z_1^*\xi_0^*+z_2\xi_1+z_2^*\xi_1^*), \eta\rangle.
\end{equation}
As all the ingredients in the right hand side of \eqref{aaa} are known, easy computations show that
\begin{equation}
\mathbbm a=\frac{k_0^2}{2\omega_c}(\omega_c-ia_ck_0^4)(-\mu_1k_0^2+\mu_2).
\end{equation}\\

In the following steps, we need to expand the equation
$\partial_t U=\mathcal L_c U +R_{11}(U)+R_{20}(U)+R_{30}(U)+\Gamma(U)$
by using the expression $U=z_1\xi_0+z_1^*\xi_0^*+z_2\xi_1+z_2^*\xi_1^*+\sum_{p+q+r+s+l+m\geq1}\Psi_{pqrslm}z_1^pz_1^{*q}z_2^rz_2^{*s}\mu_1^l\mu_2^m$ and
the normal form equations to compare coefficients.\\

 \textbf{step 2}. Compare $O(\mu_1)$ and $O(\mu_2)$ terms. For these two kinds of terms, we get 
 $$0=\mathcal L _c \Psi_{000010},\, 0=\mathcal L_c \Psi_{000001}$$
 respectively. Since $0\in\rho(\mathcal L_c)$, we can conclude that $\Psi_{000010}=0, \Psi_{000001}=0$. \\

 \textbf{step 3}. Compare $O(z_1z_1^*)$ term. We get the following relation
 $$\mathcal L_c \Psi_{110000}+2R_{20}(\xi_0, \xi_0^*)=0.$$
 Due to the form of $R_{20}$, we see easily that $R_{20}(\xi_0, \xi_0^*)=0$, which in turn gives
 $\Psi_{110000}=0$ in view of $0\in\rho(\mathcal L_c)$.\\
 
  \textbf{step 4}. Compare $O(z_1^2)$ term. The result is
  $$2i\omega_c\Psi_{200000}=\mathcal L_c \Psi_{200000}+R_{20}(\xi_0,\xi_0).$$
 Noticing that $2i\omega_c\in \rho(\mathcal L_c)$, we conclude $\Psi_{200000}=(2i\omega_c-\mathcal L_c)^{-1}R_{20}(\xi_0,\xi_0)$.
To get $\Psi_{200000}$, we first observe 
\begin{equation*}R_{20}(\xi_0, \xi_0)=\frac{\sigma''(0)}{2}\partial_x\begin{pmatrix}0\\ \exp(ik_0x)\exp(ik_0x) \end{pmatrix}
=\begin{pmatrix}0\\ ik_0\sigma''(0)\exp(2ik_0x) \end{pmatrix}.
\end{equation*}
Therefore, the system for  $\Psi_{200000}$ is given by
$$(2i\omega_c-\mathcal{L}_c)\Psi_{200000}=\begin{pmatrix}0\\ ik_0\sigma''(0)\exp(2ik_0x) \end{pmatrix}.$$
Obviously, we need to seek a solution of the form 
$$\Psi_{200000}=\exp(2ik_0x)V=\exp(2ik_0x)\begin{pmatrix}v_1\\ v_2 \end{pmatrix}.$$ 
We get the system satisfied by vector $V$
$$\Big(2i\omega_c-\begin{pmatrix}-a_c(2ik_0)^4 & 2ik_0\\ \sigma'(0)2ik_0 & -\delta_c(2ik_0)^2-(2ik_0)^4\end{pmatrix}\Big)\begin{pmatrix}v_1\\ v_2 \end{pmatrix}=\begin{pmatrix}0\\ ik_0\sigma''(0) \end{pmatrix},$$
i.e., 
$$\begin{pmatrix}2i\omega_c+16a_c k_0^4 & -2ik_0\\ -\sigma'(0)2ik_0 &2i\omega_c-4\delta_c k_0^2+16k_0^4\end{pmatrix}\begin{pmatrix}v_1\\ v_2 \end{pmatrix}=\begin{pmatrix}0\\ ik_0\sigma''(0) \end{pmatrix},$$
As the coefficient matrix above is invertible due to the fact $2i\omega_c\in\rho(\mathcal L_c)$, $V$ can be solved uniquely:
$$v_1=\frac{\sigma''(0)k_0}{-2\sigma'(0)k_0+(2i\omega_c-4\delta_c k_0^2+16k_0^4)(-i\omega_c k_0^{-1}-8a_c k_0^3)},$$
$$v_2=\frac{\sigma''(0)(\omega_c-8ia_c k_0^4)}{-2\sigma'(0)k_0+(2i\omega_c-4\delta_c k_0^2+16k_0^4)(-i\omega_c k_0^{-1}-8a_c k_0^3)}.$$
Hence we have
$$\Psi_{200000}=\exp(2ik_0x)\begin{pmatrix}\frac{\sigma''(0)k_0}{-2\sigma'(0)k_0+(2i\omega_c-4\delta_c k_0^2+16k_0^4)(-i\omega_c k_0^{-1}-8a_c k_0^3)},\\ 
\frac{\sigma''(0)(\omega_c-8ia_c k_0^4)}{-2\sigma'(0)k_0+(2i\omega_c-4\delta_c k_0^2+16k_0^4)(-i\omega_c k_0^{-1}-8a_c k_0^3)}. \end{pmatrix}.$$
\\
  
\textbf{step 5}. Compare $O(z_2z_2^*)$ term. We obtain
$$0=\mathcal L_c \Psi_{001100}+2R_{20}(\xi_1,\xi_1^*).$$
Again, we can conclude that $ \Psi_{001100}=0$ in view that $R_{20}(\xi_1,\xi_1^*)=0$ and $0\in\rho(\mathcal L_c)$.\\
  
\textbf{step 6}. Compare $O(z_1z_2)$ term. The result is
$$2i\omega_c\Psi_{101000}=\mathcal L_c\Psi_{101000}+2R_{20}(\xi_0,\xi_1),$$
which is 
$$(2i\omega_c-\mathcal L_c)\Psi_{101000}=2R_{20}(\xi_0,\xi_1).$$
We can conclude $\Psi_{101000}=0$ in view that $R_{20}(\xi_0,\xi_1)=0$ and $2i\omega_c\in\rho(\mathcal L_c)$.\\

\textbf{step 7}. Compare $O(z_1z_2^*)$ term.  We obtain
$$0=\mathcal L_c\Psi_{100100}+2R_{20}(\xi_0, \xi_1^*).$$
Noticing that
\begin{align*}
-2R_{20}(\xi_0, \xi_1^*)=-2\frac{\sigma''(0)}{2}\partial_x\begin{pmatrix}0\\ \exp(ik_0x)\exp(ik_0x) \end{pmatrix}
=\begin{pmatrix}0\\ -2ik_0\sigma''(0)\exp(2ik_0x) \end{pmatrix},
\end{align*} 
the system for $\Psi_{100100}$ is therefore given by
$$\begin{pmatrix}-a_c\partial_x^4 & \partial_x\\ \sigma'(0)\partial_x & -\delta_c\partial_x^2-\partial_x^4\end{pmatrix}\Psi_{100100}=\begin{pmatrix}0\\ -2ik_0\sigma''(0)\exp(2ik_0x) \end{pmatrix}.$$
Similarly as in the computation of $\Psi_{200000}$, we need to seek a solution of the form:
$$\Psi_{100100}=\exp(2ik_0x)W=\exp(2ik_0x)\begin{pmatrix}w_1\\ w_2 \end{pmatrix}.$$ 
In Fourier side,  we get the algebraic system for $W$
$$\begin{pmatrix}-a_c(2ik_0)^4 & 2ik_0\\ \sigma'(0)2ik_0 & -\delta_c(2ik_0)^2-(2ik_0)^4\end{pmatrix}\begin{pmatrix}w_1\\ w_2 \end{pmatrix}=\begin{pmatrix}0\\ -2ik_0\sigma''(0)\end{pmatrix},$$
i.e.,
$$\begin{pmatrix}-16a_c k_0^4 & 2ik_0\\ \sigma'(0)2ik_0 &4\delta_ck_0^2-16k_0^4\end{pmatrix}\begin{pmatrix}w_1\\ w_2 \end{pmatrix}=\begin{pmatrix}0\\ -2ik_0\sigma''(0)\end{pmatrix}.$$
As $0\in\rho(\mathcal L_c)$, we can solve $W$ uniquely
$$w_1=\frac{-\sigma''(0)}{\sigma'(0)-16a_ck_0^4(\delta_c-4k_0^2)},$$
$$w_2=\frac{8a_c\sigma''(0)k_0^3i}{\sigma'(0)-16a_ck_0^4(\delta_c-4k_0^2)}.$$
Hence we have
$$\Psi_{100100}=\exp(2ik_0x)\begin{pmatrix} \frac{-\sigma''(0)}{\sigma'(0)-16a_ck_0^4(\delta_c-4k_0^2)} \\ \frac{8a_c\sigma''(0)k_0^3i}{\sigma'(0)-16a_ck_0^4(\delta_c-4k_0^2)} \end{pmatrix}.$$
 \\
 
 Now we come back to the expansion of $\Psi(z_1, z_1^*, z_2, z_2^*, \mu_1, \mu_2)$. Due to the above steps, we can reduce the list of summand to the following:
 $$\Psi_{100010}z_1\mu_1, \Psi_{010010}z_1^*\mu_1, \Psi_{10001}z_1\mu_2, \Psi_{01001}z_1^*\mu_2,\Psi_{001010}B\mu_1,$$ 
$$\Psi_{000110}z_2^*\mu_1, \Psi_{000101}z_2\mu_2,\Psi_{000101}z_2^*\mu_2, \Psi_{200000}z_1^2, \Psi_{020000}z_1^{*2}$$
 $$ \Psi_{002000}z_2^2, \Psi_{000200}z_2^{*2}, \Psi_{100100}z_1z_2^*, \Psi_{011000}z_1^*z_2, \Psi_{111000}z_1z_1^*z_2,$$
 $$\Psi_{110100}z_1z_1^*z_2^*, \Psi_{101100}z_1z_2z_2^*, \Psi_{011100}z_1^*z_2z_2^*, \Psi_{21000}z_1^2z_1^*, etc.$$
 Another complication comes from the dependence of $\mathbbm{b,c}$ on the parameters $\mu_1, \mu_2$. We also need to expand them as $\mathbbm{b}=\mathbbm a_0 +\mathbbm a_1\mu_1+\mathbbm a_2\mu_2+h.o.t$ and $\mathbbm{c}=\mathbbm c_0 +\mathbbm c_1\mu_1+\mathbbm c_2\mu_2+h.o.t$ for comparisons as we do not exactly know \textit{a priori} how they interact with other terms.\\
 
 \textbf{step 8}. Compare $O(z_1^2z_1^*)$ term (notice that $z_1^2z_1^*=|z_1|^2z_1$). The result is
  $$\mathbbm{b}_0\xi_0+i\omega_c\Psi_{210000}=\mathcal L_c\Psi_{210000}+2R_{20}(\xi_0^*, \Psi_{200000})+3R_{30}(\xi_0,\xi_0, \xi_0^*).$$
 In general, we need to have a term $2R_{20}(\xi_0, \Psi_{110000})$ from comparisons in the right hand side of the above expression. Here we have shown that $\Psi_{110000}=0$ in \textbf{step 3}.
 All the terms in the right hand side of the above equality can be computed explicitly except the term involving $\Psi_{210000}$. In fact, we do not have a relation for $\Psi_{210000}$ to be solved explicitly through comparing the coefficients. Nevertheless, we can use duality to determine $\mathbbm b_0$, i.e., writing the above relation as
   $$\mathbbm{b}_0\xi_0+(i\omega_c-\mathcal L_c)\Psi_{210000}=2R_{20}(\xi_0^*, \Psi_{200000})+3R_{30}(\xi_0,\xi_0, \xi_0^*)$$
 and computing the inner product with $\eta\in \ker (i\omega_c-\mathcal L_c)^*$. Easy computations yield
 \begin{equation}
 \mathbbm b_0=\frac{ik_0^3\sigma''(0)^2/2\omega_c}{-2\sigma'(0) k_0+(2i\omega_c-4\delta_c k_0^2+16k_0^4)(-i\omega_c k_0^{-1}-8a_c k_0^3)}+i\frac{\sigma'''(0)k_0^2}{4\omega_c}.
 \end{equation}
 For completeness, we give some details here.
$$\langle \xi_0, \eta\rangle=1, \langle(i\omega_c-\mathcal L_c)\Psi_{210000},\eta\rangle= \langle\Psi_{210000},(i\omega_c-\mathcal L_c)^*\eta\rangle=0.$$
$$\langle 2R_{20}(\xi_0^*, \Psi_{200000}),\eta\rangle=\frac{ik_0^3\sigma''(0)/2\omega_c}{-2\sigma'(0)k_0+(2i\omega_c-4\delta_c k_0^2+16k_0^4)(-i\omega_c k_0^{-1}-8a_c k_0^3)}$$
$$\langle 3R_{30}(\xi_0, \xi_0, \xi_0^*), \eta \rangle=\frac{1}{2}\sigma'''(0)ik_0\frac{k_0}{4\pi\omega_c}2\pi=i\frac{\sigma'''(0)k_0^2}{4\omega_c},$$

\textbf{step 9}. Compare $O(z_1z_2z_2^*)$ term (notice that $z_1z_2z_2^*=z_1|z_2|^2$). We obtain
 $$\mathbbm{c}_0\xi_0+i\omega_c\Psi_{101100}=\mathcal L_c\Psi_{101100}+2R_{20}(\xi_1, \Psi_{100100})+6R_{30}(\xi_0,\xi_1, \xi_1^*).$$
 In general, we will have two additional terms $2R_{20}(\xi_0, \Psi_{001100})$, $2R_{20}(\xi_1^*, \Psi_{101000})$ in the right hand side of the above expression from comparison. Here we have shown in
 \textbf{step 5} and  \textbf{step 6} that $\Psi_{001100}=0$ and $\Psi_{101000}=0$ respectively.
 The procedure to get $\mathbbm c_0$ is similar as above. Direct computations yields
 \begin{equation}
 \mathbbm c_0=\frac{-ik_0^2\sigma''(0)^2}{2\omega_c(\sigma'(0)-16a_ck_0^4(\delta_c-4k_0^2)}+i\frac{\sigma'''(0)k_0^2}{2\omega_c}.
 \end{equation}
 
Similarly, some computational details are as follows.
$$\langle \xi_0, \eta\rangle=1, \langle(i\omega_c-\mathcal L_c)\Psi_{101100},\eta\rangle=0.$$ 
 \begin{align*}
&2R_{20}(\xi_1,\Psi_{10010}), \eta\rangle\\&=\sigma''(0)\langle \begin{pmatrix} 0\\ \partial_x\Big( \exp(-ik_0x)\frac{-\exp(2ik_0x)\sigma''(0)}{\sigma'(0)-16a_ck_0^4(\delta_c-4k_0^2)}  \Big ) \end{pmatrix}, \frac{k_0}{4\pi \omega_c}\exp(ik_0x)\begin{pmatrix} -ia_c k_0^3+\frac{\omega_c}{k_0}\\1\end{pmatrix}\rangle\\
&=\frac{-ik_0^2\sigma''(0)^2}{2\omega_c(\sigma'(0)-16a_ck_0^4(\delta_c-4k_0^2))},
\end{align*}
 $$\langle 6R_{30}(\xi_0, \xi_1, \xi_1^*), \eta\rangle=\sigma'''(0)ik_0\frac{k_0}{4\pi\omega_c}2\pi=i\frac{\sigma'''(0)k_0^2}{2\omega_c}.$$ This completes the current step.\\
 
Now we come back again to collect the information gathered up to now that is related to the parameters $\mathbbm a$, $\mathbbm b_0$ and $\mathbbm c_0$. The real and imaginary parts of them will be important for us, and are computed as listed below:
\begin{equation}\label{mathbbma}
\mathbbm a_r=\frac{k_0^2}{2}(-\mu_1 k_0^2+\mu_2),\, \mathbbm a_i=\frac{a_c k_0^6}{2\omega_c}(-\mu_1 k_0^2+\mu_2)
\end{equation}

\begin{equation}\label{mathbbmbr}
\mathbbm b_{0r}=-6k_0^6\sigma''(0)^2\delta_c\alpha^{-1},
\end{equation}
\begin{equation}\label{mathbbmbi}
\mathbbm b_{0i}=k_0^4\sigma''(0)^2(-\sigma'(0)k_0^2\omega_c^{-1}+\omega_c-16k_0^{8}\omega_c^{-1}a_c(3-a_c))\alpha^{-1}
\end{equation}
where $\alpha=\alpha(k_0^2)=\Big(-2\sigma'(0)k_0^2+2\omega_c^2-32k_0^8a_c(3-a_c)\Big)^2+144k_0^4\omega_c^2\delta_c^2$.
 \begin{equation}\label{mathbbmc}
 \mathbbm c_{0r}=0,\,\, \mathbbm c_{0i}=\frac{-k_0^2\sigma''(0)^2}{2\omega_c(\sigma'(0)-16a_ck_0^4(\delta_c-4k_0^2))}+\frac{\sigma'''(0)k_0^2}{2\omega_c}.
 \end{equation}\\

\textbf{step 10}. We conclude that the dynamics on the center manifold(s) of the system \eqref{bifurcation equation} has the following form

\begin{equation}\label{normal}
\begin{cases}
\frac{d}{dt}z_1=i\omega_c z_1+ z_1\Big(\mathbbm{a}(\mu_1,\mu_2)+\mathbbm{b}_0|z_1|^2+\mathbbm{c}_0|z_2|^2+O(|\mu_1|,|\mu_2|, |z_1|, |z_2|)\Big)\\
\frac{d}{dt}z_2=i\omega_c z_2+z_2\Big(\mathbbm{a}(\mu_1,\mu_2)+\mathbbm{b}_0|z_2|^2+\mathbbm{c}_0|z_1|^2)+O(|\mu_1|,|\mu_2|, |z_1|, |z_2|)\Big)
\end{cases}
\end{equation}
where $\mathbbm b_0$ and $\mathbbm c_0$ do not depend on $\mu_1$ or $\mu_2$. The term $O(|\mu_1|,|\mu_2|, |z_1|, z_2)$ is a sum of terms $O((|\mu_1|+|\mu_2|)(|z_1^2|+|z_2^2|))$, $O((|\mu_1^2|+|\mu_2^2|)(|z_1^2|+|z_2^2|))$ and $O(|z_1|^4+|z_2|^4)$. Terms like
$O((|\mu_1|+|\mu_2|)(|z_1^2|+|z_2^2|))$ and $O((|\mu_1^2|+|\mu_2^2|)(|z_1^2|+|z_2^2|))$ stem from the fact that $\mathbbm{b,c}$ depend on $\mu_1$ and $\mu_2$. It is important to notice that
the terms $O((|\mu_1|+|\mu_2|)(|z_1^2|+|z_2^2|))$ do play a role though not a dominant one in the analysis of bifurcation dynamics. 
 
The analysis of the finite dimensional dynamical system obtained in equation \ref{normal} is routine. However, we shall be aware of the subtle differences of the current case from the standard situations. To this end, we
need to study the truncated normal form \eqref{normal} at cubic order following
\cite{Ch, HI, IA}.  Here our analysis refines that of \cite{HI} (see page 133) in the current two bifurcation parameters situation.
We will see that the signs of the parameters $\mathbbm a$, $\mathbbm b_0$ and $\mathbbm c_0$ play the key role for the bifurcation dynamics.
To proceed, we introduce the polar coordinates $z_1=r_1\exp(i\theta_1)$ and $z_2=r_2\exp(i\theta_2)$,
the truncated normal form at cubic order becomes
\begin{equation}\label{rd}
\begin{cases}
\frac{dr_1}{dt}=r_1 \Big(\mathbbm{a}_r+(\mathbbm{b}_{0r}+O_r(|\mu_1|+|\mu_2|)) r_1^2 +(\mathbbm{c}_{0r}+O_r(|\mu_1|+|\mu_2|))  r_2^2\Big),\\
\frac{dr_2}{dt}=r_2 \Big(\mathbbm{a}_r+(\mathbbm{b}_{0r}+O_r(|\mu_1|+|\mu_2|)) r_2^2 +(\mathbbm{c}_{0r}+O_r(|\mu_1|+|\mu_2|)) r_1^2\Big),\\
\frac{d\theta_1}{dt}=\omega_c+\mathbbm{a}_i +(\mathbbm{b}_{0i}+O_i(|\mu_1|+|\mu_2|)) r_1^2+(\mathbbm{c}_{0i}+O_i(|\mu_1|+|\mu_2|)) r_2^2,\\
\frac{d\theta_2}{dt}=\omega_c+\mathbbm{a}_i +(\mathbbm{b}_{0i}+O_i(|\mu_1|+|\mu_2|)) r_2^2+(\mathbbm{c}_{0i}+O_i(|\mu_1|+|\mu_2|)) r_1^2,\\
\end{cases}
\end{equation}
In the above, $\mathbbm{b}_{0r}+O_r(|\mu_1|+|\mu_2|)$ and $\mathbbm{c}_{0r}+O_r(|\mu_1|+|\mu_2|)$ are nothing else but the real parts of $\mathbbm{b,c}$ respectively. Similarly, $\mathbbm{b}_{0i}+O_i(|\mu_1|+|\mu_2|)$ and $\mathbbm{c}_{0i}+O_i(|\mu_1|+|\mu_2|)$ are the imaginary parts of $\mathbbm{b,c}$ respectively.

The equations on $(r_1, r_2)$ and $(\theta_1, \theta_2)$ decouple. However, we can not conclude in general by $\mathbbm{c}_{0r}=0$ that the equations on $r_1$ and $r_2$ also decouples due to the coefficients $O_r(|\mu_1|+|\mu_2|)$. The point is that we can first seek bifurcated equilibria in the radial equations by assuming $r_1\equiv0$ or $r_2\equiv0$.
Consider the parameter $\vartheta=-\mu_1 k_0^2+\mu_2$. Of course $\vartheta$ depends on $k_0$. We will call the lines such that $\vartheta=0$ lines of degeneracy. Each pair $\pm k_0$ corresponds to one of these lines. In the $(\mu_1, \mu_2)$-plane, these lines with slopes $k_0^2$ and passing the first and third quadratures divide the coordinate plane into half planes. For us, we always do analysis in a small neighborhood of $(\mu_1,\mu_2)=(0,0)$ in $\mathbbm R^2$. Therefore, these lines of degeneracy divide the neighborhood into halves. 

Next, we let $k_0$ be arbitrary but fixed and do a analysis away from the line of degeneracy associated with $k_0$. 
Consider the auxiliary real function $f(r)=\frac{k_0^2}{2}\vartheta+(b+O(|\vartheta|))r^2$ defined on $[0,\epsilon]$ for some small $\epsilon>0$. If $\vartheta b>0$, $f$ has the only trivial root $r=0$. If
$\vartheta b<0$, $f$ has a nontrivial root $r=\sqrt{\frac{-k_0^2\vartheta}{2(b+O(|\vartheta|))}}=O(|\vartheta|^{1/2})=O((|\mu_1|+|\mu_2|)^{1/2})$. We emphasize that $|\vartheta|$ is small. Now, consider the right hand sides of the radial equations of $r_1$ and $r_2$. If we let $r_2\equiv 0$, then the right hand side of equation on $r_1$ contains a factor with the same structure as $f(r)$. We can consider the $r_2$ equation similarly. Hence the above analysis for $f(r)$ applies.  Let $r_{*}(\vartheta)=r_{*}(\mu_1,\mu_2)=\sqrt{\frac{-k_0^2\vartheta}{2(\mathbbm b_{0r}+O(|\vartheta|))}}$. We conclude that besides the trivial solution $(0,0)$, there are bifurcated solutions of the forms $(r_{*}(\vartheta), 0)$, and $(0, r_{*}(\vartheta))$. Further, if $\mathbbm b_{0r}+\mathbbm c_{0r}$ does not vanish, we may consider the situation $r_1\equiv r_2$. In such a situation, the $r_1$ and $r_2$ equations are the same and both contain a factor of the form $f(r)$ in the right hand sides. Notice that we have $\mathbbm c_{0r}=0$. Therefore, there are bifurcated solutions of the form $( r_{*}(\vartheta), r_{*}(\vartheta))$. We are in a position to analyze the two angular equations, which is easy now. From above analysis, we know that all the three families of bifurcated solutions have magnitude $O(|\vartheta|^{1/2})$. As a consequence, we could arrange that $\frac{d\theta_1}{dt}\geq\frac{\omega_c}{2}>0$ and $\frac{d\theta_2}{dt}\geq\frac{\omega_c}{2}>0$ when $|\vartheta|$ remains small, which enables us to conclude that all the three families of bifurcated equilibria correspond to genuine periodic waves of the system \eqref{rd}. The equilibria $(r_{*}(\vartheta), 0)$ and $(0, r_{*}(\vartheta))$ correspond to rotating waves on $r_1$-axis and $r_2$-axis, which is the same as for the Hopf bifurcation with $SO(2)$ symmetry. The symmetry $S$ plays the role of exchanging the two axes, i.e., exchanging the rotating waves corresponding to $r_2=0$ into the rotating waves corresponding to $r_1=0$. The equilibria $(r_{*}(\vartheta), r_{*}(\vartheta))$ with $r_1=r_2$ correspond to standing waves, another class of bifurcating periodic solutions.  These waves correspond to a torus of solutions of the normal form
\begin{align*}
V_0(t,\mu_1,\mu_2, \delta_0,\delta_1)&=r_{*}(\mu_1,\mu_2)\Big( \exp(i\omega_{*}(\mu_1,\mu_2)t+\delta_0)\xi_0+ \exp(i\omega_{*}(\mu_1,\mu_2)t+\delta_1)\xi_1\Big)\\
&+r_{*}(\mu_1,\mu_2)\Big( \exp(-i\omega_{*}(\mu_1,\mu_2)t+\delta_0)\xi_0^*+ \exp(i\omega_{*}(\mu_1,\mu_2)t+\delta_1)\xi_1^*\Big)
\end{align*}
for any $(\delta_1,\delta_2)\in\mathbb{R}^2$, which induces a torus
of solutions $U(t,\mu_1,\mu_2, \delta_1,\delta_2)$ in $Y$ of the nonlinear
perturbation system \eqref{perturb}. The $\omega_{*}(\mu_1,\mu_1)$ is the phase
function determined by the $(\theta_1, \theta_2)$ equation in system
\eqref{rd} such that $\omega_{*}(0, 0)=\omega_c$. These standing waves in
addition possess the following symmetry
$$T_{\delta_1-\delta_0}SU(t,\mu_1,\mu_2, \delta_0,\delta_1)=U(t,\mu_1,\mu_2, \delta_0,\delta_1), \,\, T_{2\pi}U(t,\mu_1,\mu_2, \delta_0,\delta_1)=U(t,\mu_1,\mu_2 ,\delta_0,\delta_1),$$
$$T_{\pi}U(t,\mu_1,\mu_2, \delta_0,\delta_1)=U(t+\frac{\pi}{\omega_{*}(\mu_1,\mu_2)}, \mu_1,\mu_2 \delta_0, \delta_1), \,\, SU(t,\mu_1,\mu_2, \delta_0, \delta_0)=U(t,\mu_1,\mu_2, \delta_0, \delta_0).$$
The analysis of the stability of the three families of bifurcated waves are straightforward by examining the sign of $\mathbbm b_{0r}$ in which the roles of the two numbers $\sigma''(0)$ and $\delta_c$ are essential. Another striking fact is that $\sigma''(0)$ enters the reduced dynamics on the center manifold(s) through $\mathbbm b_0$ and $\mathbbm c_0$ in the form of $\sigma''(0)^2$. Now we complete
the analysis of the bifurcation dynamics and the paper.

\end{document}